\documentclass[10pt,twoside,a4paper]{article}

\usepackage[utf8]{inputenc}
\usepackage[T1]{fontenc}

\usepackage{hyperref}
\usepackage{lmodern}
\usepackage{amsmath,amssymb,amsfonts,amsthm,mathtools,mathrsfs,pifont}
\usepackage{enumitem}
\usepackage{fancyhdr}
\usepackage{ifthen}

\newtheorem{theorem}{Theorem}[section]
\newtheorem{lemma}[theorem]{Lemma}
\newtheorem{prop}[theorem]{Proposition}

\theoremstyle{remark}

\theoremstyle{definition}

\usepackage{tikz}

\usetikzlibrary{calc}
\usetikzlibrary{math}
\usetikzlibrary{patterns}
\usetikzlibrary{arrows.meta,decorations.shapes}

\definecolor{lightgray}{RGB}{211,211,211}
\definecolor{lightgreen}{RGB}{144, 238, 144}
\definecolor{llightgray}{RGB}{237,237,237}
\definecolor{lightyellow}{RGB}{205,201,165}
\definecolor{lightskincol}{RGB}{255,229,200}
\definecolor{skincol}{RGB}{255,195,170}
\definecolor{darkskincol}{RGB}{240,184,160}
\definecolor{skinpinkcol}{RGB}{255,204,203}
\definecolor{darkkhaki}{RGB}{204,204,0}
\definecolor{lightorange}{RGB}{255,186,102}
\definecolor{darkorange}{RGB}{255,140,0}
\definecolor{mediumturquoise}{RGB}{72, 209, 204}

\usepackage[left=3.5cm,right=3.5cm]{geometry}

\newcommand{\abs}[1]{\left| #1 \right|}
\newcommand{\norm}[1]{\left\lVert #1 \right\rVert}
\newcommand{\scp}[1]{\left\langle #1 \right\rangle}
\newcommand{\set}[1]{\left\lbrace #1\right\rbrace}

\newcommand{\B}{\mathbb{B}}
\newcommand{\C}{C}
\newcommand{\Hilbert}{\mathcal{H}}
\newcommand{\K}{\mathcal{K}}
\newcommand{\M}{\mathcal{M}}

\newcommand{\Radon}{\mathcal{R}}

\newcommand{\restr}[2]{\left. #1 \right|_{#2}}

\newcommand{\N}{\mathbb{N}}
\newcommand{\R}{\mathbb{R}}
\newcommand{\Sp}{\mathbb{S}}

\newcommand{\nd}{\partial_{\nu}}

\renewcommand{\div}{\mathrm{div}}

\ifpdf
\hypersetup{
  pdftitle={Explicit inversion formulas for the two-dimensional wave equation from Neumann traces},
  pdfauthor={F. Dreier, and M. Haltmeier}
}
\fi

\usepackage{soul}
\usepackage{xcolor}
\colorlet{lred}{red!40}
\colorlet{lgreen}{green!40}
\colorlet{lblue}{blue!40}

\numberwithin{equation}{section}
\numberwithin{theorem}{section}

\allowdisplaybreaks

\begin{document}
	\title{{\fontsize{12}{15}\bfseries\uppercase{Explicit inversion formulas for the two-dimensional wave equation from Neumann traces}}\thanks{\textbf{Funding}: This work has been supported by the Austrian Science Fund (FWF), project P 30747-N32.}}
	
	
\author{Florian Dreier\thanks{Department of Mathematics, University of Innsbruck, Technikerstraße 13, A-6020 Innsbruck, Austria (Florian.Dreier@uibk.ac.at, Markus.Haltmeier@uibk.ac.at).}
\and and \and Markus Haltmeier\footnotemark[2]}

	\date{}
	
	\maketitle

	\begin{abstract}
		In this article we study the problem of recovering the initial data of the two-dimensional wave equation from Neumann measurements on a convex domain $\Omega\subset\R^2$ with smooth boundary. We derive an explicit inversion formula of a so-called back-projection type and deduce exact inversion formulas for circular and elliptical domains. In addition, for circular domains, we show that the initial data can also be recovered from any linear combination of its solution and its normal derivative on the boundary. Numerical results of our implementation of the derived inversion formulas are presented demonstrating their accuracy  and stability.
	\end{abstract}
	\vspace{0.375cm}\par
	\vspace{0.375cm}
	
	
	\section{Introduction}
Inverse source problems for the wave equation form the basis of various practical applications. For example, in photoacoustic tomography (PAT), one is interested to recover the initial pressure distribution from measurements of the induced pressure waves outside of the investigated object. In PAT and other applications, the corresponding initial value problem is given by the wave equation
	\begin{equation}
		\label{eq:waveeq}
		\begin{aligned}
	    	(\partial_t^2-\Delta)u(x,t)&=0 &\quad &\text{for } (x,t) \in \R^n \times (0,\infty),\\
	    	u(x,0)&=f(x)&\quad &\text{for } x \in \R^n,&\\
	    	(\partial_tu)(x,0)&=0&\quad &\text{for } x \in \R^n,
	    \end{aligned}
	\end{equation}
	where $\Delta$ denotes the Laplacian in the spatial component, $f\colon \R^d\to\R$ the initial pressure distribution and $n\geq 2$ the spatial dimension. A typical inverse problem arising in PAT is as follows: Given a bounded domain $\Omega\subset\R^n$ with smooth boundary $\partial\Omega$ and the solution of the wave equation \eqref{eq:waveeq} on $\partial\Omega\times(0,\infty)$, determine the function $f$ in \eqref{eq:waveeq} where $f$ is assumed to be a smooth function vanishing outside of $\Omega$  \cite{LiWa09,WaHu12,RosNtzRaz13}.
	
	The problem of recovering $f$ from Dirichlet data $\restr{u}{\partial\Omega\times(0,\infty)}$ (Dirichlet trace) has been extensively studied in the recent years. Existing techniques for the reconstruction of $f$ are based on Fourier domain algorithms \cite{AK,haltmeier2007thermo,kunyansky2007series,xu2002exact}, iterative algorithms \cite{HalNgu17,belhachmi2016direct,huang2013full,arridge2016adjoint,DeaBueNtzRaz12,SchPerHal18}, time reversal \cite{BurMatHalPal07,HriKucNgu08,SteUhl09,nguyen2016dissipative} and explicit inversion formulas of the back-projection type \cite{FinPatRak04,Hal13,Nat12,Hal14,Kun07,Ngu09,AnsFibMadSey12,Pal14,Sal14}. In this paper, we focus on the latter class, which provides theoretical inside and serves as basis of efficient numerical algorithms. Inversion formulas for Dirichlet measurements are well-known for special domains. In \cite{FinPatRak04,FinHalRak07} several exact formulas for spherical  domains in odd \cite{FinPatRak04} and even \cite{FinHalRak07} dimension have been derived, whereas in \cite{Hal13,Nat12,Hal14}  inversion formulas for elliptical domains in the plane \cite{Hal13}, 3D space \cite{Hal14} and arbitrary spatial dimension \cite{Hal14} have been  derived. In \cite{HalSer15a,HalSer15b} the problem of determining initial data from knowledge of Dirichlet measurements on certain quadratic hypersurfaces, including parabolic surfaces, has also been analyzed. In \cite{Kun11,Kun15}, explicit inversion formulas for the wave equation on the surface of certain polygons, polyhedra and corner-like domains are presented.

Typical ultrasound sensors are often direction dependent and the measured data is a weighted combination of the normal derivative and the pressure \cite{WaPaKuXiStWa03,WaHu12}.
Such measurements can be modeled by
\begin{equation}
		au(x,t) + b\nd u(x,t),\quad (x,t)\in\partial\Omega\times(0,\infty),
	\label{eq:mixedmeas}
\end{equation}
where $\nd$ denotes the normal derivative of $u\colon \R^n\times[0,\infty)\to\R$ along $\partial\Omega$ with respect to the spatial variable and $a,b\neq 0$ are constants. In this paper we study the problem of reconstructing  $f$ in \eqref{eq:waveeq}  from data \eqref{eq:mixedmeas}.
The case $b=0$ corresponds to Dirichlet measurements and is the standard case studied in photoacoustic tomography. The case $a=0$ corresponds to Neumann measurements
\begin{equation}
	b\nd u(x,t),\quad (x,t)\in\partial\Omega\times(0,\infty).
	\label{eq:neumannmeas}
\end{equation}
To this day, explicit formulas for the inversion of the wave equation from Neumann measurements are hardly known. For mixed measurements \eqref{eq:mixedmeas}, in \cite{zangerl2018photoacoustic},  series inversion formulas have been derived. In \cite{FinRak07}, an exact inversion formula in 3D of universal back-projection type for recovering the initial data $g\in C_c^\infty(\B^3_\rho)$ in
\begin{equation*}
		\begin{aligned}
	    	(\partial_t^2-\Delta)v(x,t)&=0 &\quad &\text{for } (x,t) \in \R^3 \times (0,\infty),\\
	    	v(x,0)&=0&\quad &\text{for } x \in \R^3,&\\
	    	(\partial_tv)(x,0)&=g(x)&\quad &\text{for } x \in \R^3
	    \end{aligned}
\end{equation*}
from Neumann measurements on the boundary of the open origin centered ball $\B^3_\rho\subset\R^3$ with radius $\rho>0$ has been derived.

In the present paper, we establish explicit formulas for the inversion of the two-dimensional wave equation \eqref{eq:waveeq} from Neumann measurements $\restr{\nd u}{\partial\Omega\times(0,\infty)}$ (Neumann trace). In Section~\ref{subsec:explformconv}, we derive an explicit formula for the initial data up to a smoothing integral operator for convex domains by the knowledge of \eqref{eq:neumannmeas}. We will show that the smoothing integral operator vanishes for circular and elliptical domains (see \ref{subsec:exactform}). In Section~\ref{subsec:exactformneumann}, we derive an inversion formula for circular domains that exactly recovers $f$ from mixed measurements given in \eqref{eq:mixedmeas}. Numerical results with the derived formulas are presented
in Section~\ref{sec:num}.

\section{Notation and auxiliary results}

\subsection{Notation}

Before we present our results, we start with the notation which we are using throughout this article. By $u\colon \R^2\times[0,\infty)\to\R$ we denote the solution of the two-dimensional wave equation \eqref{eq:waveeq} with initial data $(f,0)$ and by $v\colon \R^2\times[0,\infty)\to\R$ the solution of the wave equation
	\begin{equation}
		\label{eq:waveeq2}
		\begin{aligned}
	    	(\partial_t^2-\Delta)v(x,t)&=0 &\quad &\text{for } (x,t) \in \R^2 \times (0,\infty),\\
	    	v(x,0)&=0&\quad &\text{for } x \in \R^2,&\\
	    	(\partial_tv)(x,0)&=g(x)&\quad &\text{for } x \in \R^2
	    \end{aligned}
	\end{equation}
	with initial data $(0,g)$, where $f,g\in\C_c^\infty(\Omega)$.
	
By $\nu(x)\in\R^2$ we denote the outward unit normal vector of $\partial\Omega$ in $x\in\partial\Omega$.	
	As already mentioned in the introduction, $\nd u(x,t)$ denotes the normal derivative of $u$ at $(x,t)\in \partial\Omega\times(0,\infty)$ with respect to the spatial variable $x$, that is, the derivative of the function
	\begin{equation}	
		\R\to\R\colon h\mapsto u(x+h\nu(x),t)
		\label{eq:normderiv}
	\end{equation}
at zero. Applying the chain rule in \eqref{eq:normderiv} we see that $\nd u(x,t)=\scp{\nabla u(x,t),\nu(x)}$ where $\nabla u(x,t)\in\R^2$ is the gradient of $u$ with respect to $x$.
	
For an integrable function $f\colon \R^2\to\R$, the spherical mean operator of $f$ is defined as \[\M f\colon \R^2\times [0,\infty)\to \R\colon (x,t)\mapsto \frac{1}{2\pi}\int_{\Sp^1} f(x+r\omega) d\sigma(\omega),\] where $\Sp^1\coloneqq\set{x\in\R^2\mid \norm{x}=1}$,  with $\norm{\, \cdot\, }$ being the Euclidian distance is the unit circle, and $\sigma$ the standard surface measure on manifolds. The Radon transform of $f$ is defined in a similar way \[\Radon f\colon \Sp^1\times \R\to \R\colon (\theta,s)\mapsto \int_\R f(s\theta+a\theta^\perp)da,\] where $\theta^\perp$ is a unit vector orthogonal to $\theta$.
By \[\Hilbert_s \varphi\colon \Sp^1\times\R\to\R\colon (\theta,s)\mapsto \frac{1}{\pi}\lim_{\varepsilon\searrow 0}\int_{(s-\varepsilon,s+\varepsilon)}\frac{\varphi(\theta,t)}{s-t}dt\] we denote the Hilbert transform of a function $\varphi\colon \Sp^1\times\R\to\R$ in the second variable. If $\varphi(\theta,\cdot)$ is differentiable for some $\theta\in\Sp^1$, then we also use $\partial_s \varphi(\theta,\cdot)$ to denote the derivative of $\varphi$ with respect to the second variable.

The general inversion formula that we derive for an elliptic domain will
 recover any smooth function $f\in C_c^\infty(\Omega)$
up to smoothing integral operator
	\begin{equation} \label{eq:K}
		\K_\Omega f (x)\coloneqq \frac{1}{8\pi^2}\int_\Omega f(y)\frac{\left(\partial_s^2\Hilbert_s\Radon \chi_\Omega\right)\left(\tilde{n}(x,y),\tilde{s}(x,y)\right)}{\norm{x-y}}dy,
	\end{equation}
which coincides with the error operator appearing in \cite{Hal13}.
Here $\chi_\Omega\colon \R^2\to\R$ is the indicator function of $\Omega$ in $\R^2$, and
for  $x, y\in\R^2$ with $x\neq y$ we set $\tilde{n}(x,y)\coloneqq (y-x)/\norm{x-y}$, $ \tilde{s}(x,y)\coloneqq (\norm{y}^2-\norm{x}^2)/(2\norm{x-y})$.

\subsection{Auxiliary results}
	
First, we recall the well-known explicit solution formulas (see \cite{Jo82})
		\begin{align}
			\label{eq:wavesol1}
			u(x,t)&=\partial_t \int_0^t \frac{r\M f(x,r)}{\sqrt{t^2-r^2}}dr,\quad (x,t)\in\R^2\times[0,\infty),
		\\	\label{eq:wavesol2}
			v(x,t)&=\int_0^t \frac{r\M g(x,r)}{\sqrt{t^2-r^2}}dr,\quad (x,t)\in\R^2\times[0,\infty),
		\end{align}
where  $f,g\in C_c^\infty(\Omega)$, $u$ is the solution of \eqref{eq:waveeq}
with initial data  $(f,0)$ and $v$ the solution of  \eqref{eq:waveeq2} with initial data
$(0,g)$.

The first Lemma which we use for derivation of a explicit inversion formulas is an   integral identity for the spherical means. It  is a corrected  version of \cite[Lemma~{2.1}]{Hal13},
where the term $\int_\Omega \varphi(y)f(y)\log (\norm{x-y}^2 )dy$ is missing.

	\begin{lemma}
		\label{lem:lemidentity1}
		Let $f\in C_c^\infty(\Omega)$ and $\varphi\colon \R^2\to\R$ a continuously differentiable function on $\overline{\Omega}$ vanishing outside of $\overline{\Omega}$. Then, for every $x\in\Omega$ the following identity holds
		\begin{equation}
			\begin{aligned}
			&\int_{\R^2}\varphi(y)\left(\int_0^\infty (\partial_r \M f)(y,r)\log\abs{r^2-\norm{x-y}^2} dr\right)dy\\
			&=\frac{1}{2\pi}\int_{\R^2}\frac{f(y)}{\norm{x-y}}\left(\int_\R\left(\partial_s\Radon\varphi\right)\left(\tilde{n}(x,y),s\right)\log\abs{2\norm{x-y}\left(s-\tilde{s}(x,y)\right)} ds\right)dy\\
			&\quad -\int_\Omega \varphi(y)f(y)\log\left(\norm{x-y}^2\right)dy.
			\end{aligned}
			\label{eq:lemidentity1}
		\end{equation}
	\end{lemma}
	\begin{proof}
		\begin{enumerate}[wide=\parindent,label=(\roman*)]			
			\item \label{item:lemidentityphi} First we show identity \eqref{eq:lemidentity1} when $\log\abs{\,\cdot\,}$ is replaced by a differentiable and integrable function $\Phi\colon\R\to\R$.
			
			Applying integration by parts on the left integral in \eqref{eq:lemidentity1} yields
			\begin{align*}
				\int_0^\infty &(\partial_r \M f)(y,r)\Phi\left(\abs{r^2-\norm{x-y}^2}\right) dr\\
				&=-f(y)\Phi\left(-\norm{x-y}^2\right)-2\int_0^\infty \M f (y,r)\Phi'\left(r^2-\norm{x-y}^2\right)rdr.
			\end{align*}
			Using the definition of the spherical mean operator and introducing polar coordinates, we further obtain
			\begin{align*}
				&\int_0^\infty (\partial_r \M f)(y,r)\Phi\left(\abs{r^2-\norm{x-y}^2}\right) dr\\
				&=-f(y)\Phi\left(-\norm{x-y}^2\right)-\frac{1}{\pi}\int_0^\infty \int_{\Sp^1} f(y+r\omega,r)\Phi'\left(r^2-\norm{x-y}^2\right)rd\sigma(\omega)dr\\
				&=-f(y)\Phi\left(-\norm{x-y}^2\right)-\frac{1}{\pi}\int_{\R^2} f(y+z)\Phi'\left(\norm{z}^2-\norm{x-y}^2\right)dz.
			\end{align*}
			After substituting $z$ with $z-y$ in the last integral,
			multiplying with $\varphi(y)$ and integrating both sides with respect to $y$, we obtain
			\begin{align*}
				\int_{\R^2} \varphi(y)&\int_0^\infty(\partial_r \M f)(y,r)\Phi\left(\abs{r^2-\norm{x-y}^2}\right) dr dy\\
				&=-\int_{\R^2}\varphi(y)f(y)\Phi\left(-\norm{x-y}^2\right)dy-\frac{1}{\pi}\int_\Omega \varphi(y) f(y) \int_{\R^2} \Phi'\left(\norm{z}^2-\norm{x-y}^2\right)dz dy,
			\end{align*}
			where we used Fubini's theorem to change the order of the integrals. In \cite{Hal13} it has been shown that second integral on the hand right side equals to
			\begin{equation*}
				\frac{1}{2\pi}\int_{\R^2}\frac{f(y)}{\norm{x-y}}\left(\int_\R\left(\partial_s\Radon\varphi\right)\left(\tilde{n}(x,y),s\right)\Phi\left(2\norm{x-y}\left(s-\tilde{s}(x,y)\right)\right) ds\right)dy,
			\end{equation*}
			which shows the desired identity for $\Phi$ in place of $\log\abs{\,\cdot\,}$.
			\item Finally, we choose a sequence $(\Phi_n)_{n\in\N}$ of differentiable and integrable functions converging pointwise to $\log\abs{\,\cdot\,}$. Thus, applying \ref{item:lemidentityphi} and Lebesgue's dominated convergence theorem yield the claimed identity.\qedhere
		\end{enumerate}
	\end{proof}
The next Lemma can  be found in \cite{Hal13}, where the proof however
contains a small error. Below  we give a  corrected proof based on the corrected Lemma~\ref{lem:lemidentity1}, which for the convenience of the reader is given in full detail.

	\begin{lemma}
		\label{lem:lemidentity2}
		Suppose that $f,g\in C_c^\infty(\Omega)$. Then, we have
		\begin{equation*}
			\int_\Omega\int_0^\infty u(x,t)v(x,t) dt dx=-\frac{1}{8\pi^2}\int_\Omega\int_\Omega f(x)g(y)\frac{\left(\Hilbert_s\Radon \chi_\Omega\right)\left(\tilde{n}(x,y),\tilde{s}(x,y)\right)}{\norm{x-y}}dxdy,
		\end{equation*}
		where $u$ is the solution of \eqref{eq:waveeq} with initial data $(f,0)$ and $v$ the solution of \eqref{eq:waveeq2} with initial data $(0,g)$. Moreover, $\tilde{n}(x,y)\coloneqq (y-x)/\norm{x-y}$, $ \tilde{s}(x,y)\coloneqq (\norm{y}^2-\norm{x}^2)/(2\norm{x-y})$ for $x \neq y \in \Omega$.
	\end{lemma}
	\begin{proof}
		\begin{enumerate}[wide=\parindent,label=(\roman*)]
			\item First, we show that $u(x,\cdot)v(x,\cdot)$ is integrable on $(0,\infty)$ for any fixed $x\in\Omega$. By applying the Leibnitz-rule and  \eqref{eq:wavesol1} we obtain
			\begin{equation}
				\label{eq:wavesol3}
				u(x,t)=f(x)+\int_0^t \frac{t}{\sqrt{t^2-r^2}}\partial_r\M f(x,r) dr,\quad t\in[0,\infty).
			\end{equation}
			Since $\Omega$ is bounded, we can find $R>0$ such that $\partial_r\M f(x,t)=\M f(x,t)=0$ for $r\geq R$. Now, let $T>R$ be a fixed positive number. From the above equality we deduce
			\begin{align*}
				\abs{u(x,t)}&\leq \abs{f(x)}+t\norm{\partial_r \M f}_\infty \int_0^t \frac{1}{\sqrt{t^2-r^2}}dr=\abs{f(x)}+t\norm{\partial_r\M f}_\infty \frac{\pi}{2}
			\end{align*}
			for $0<t\leq T$. For $t>T$, we obtain the inequality
			\begin{align*}
				\abs{u(x,t)}&=\abs{\partial_t \int_0^R \frac{r\M f(x,r)}{\sqrt{t^2-r^2}}dr}\leq t\norm{\partial_r\M f}_\infty\frac{R}{\left(t^2-T^2\right)^{3/2}}.
			\end{align*}
			by differentiating under the integral sign in \eqref{eq:wavesol1}. Analogously, we deduce from \eqref{eq:wavesol2} the estimates
			\begin{equation*}
			\abs{v(x,t)}\leq
			\begin{cases}
				 t\norm{\partial_r\M g}_\infty
				 & \text{ if }  0<t\leq T
				 \\
				\norm{\partial_r\M g}_\infty\frac{R^2}{2\sqrt{t^2-R^2}}
				 & \text{ if }  t>T \,.			
				 \end{cases}
\end{equation*}
			Thus, we obtain
			\begin{equation*}
				\int_0^\infty \abs{u(x,t)v(x,t)}dt=\int_0^T \abs{u(x,t)v(x,t)}+\int_T^\infty \abs{u(x,t)v(x,t)}\leq c_1+c_2\int_{T^2-R^2}^\infty \frac{1}{t^2} dt < \infty \,,
			\end{equation*}
			where $c_1,c_2$ are positive constants. This shows the integrability of $u(x,\cdot)v(x,\cdot)$ and in particular the integrability of $uv$ on $\Omega\times (0,\infty)$.
			\item \label{item:step1}In the next step we show that
			\begin{equation}
			\label{eq:identity1}
			\begin{aligned}
				&\int_\Omega \int_0^\infty u(x,t)v(x,t)dt dx\\
				&=-\int_\Omega\int_0^\infty f(x)\log(r_1)r_1\M g(x,r_1)dr_1 dx\\
				&\quad -\frac{1}{2}\int_\Omega\int_0^\infty\int_0^\infty\partial_r\M f(x,r_0)r_1\M g(x,r_1)\log(\abs{r_1^2-r_0^2})dr_0 dr_1 dx \,.
			\end{aligned}
			\end{equation}
			For that purpose, let $T>0$ be fixed again. From \eqref{eq:wavesol3} we see that
			\begin{multline*}
			\int_0^T u(x,t)v(x,t)dt dx
			=
				\int_0^Tf(x)v(x,t)dt\\
				+\int_0^T\int_0^T\int_0^T\chi_{(0,t)^2}(r_0,r_1)\frac{t\partial_r\M f(x,r_0)r_1\M f(x,r_1)}{\sqrt{t^2-r_0^2}\sqrt{t^2-r_1^2}}dr_1dr_0 dt.
			\end{multline*}
			The right triple-integral can be evaluated to
			\begin{multline} \label{eq:integral1}
				\int_0^T\int_0^T \partial_r\M f(x,r_0)r_1\M g(x,r_1)\log\left(\sqrt{T^2-r_0^2}+\sqrt{T^2-r_1^2}\right)dr_1dr_0\\
				-\frac{1}{2}\int_0^T\int_0^T \partial_r\M f(x,r_0)r_1\M g(x,r_1)\log{\abs{r_1^2-r_0^2}}dr_1dr_0\,,
			\end{multline}
			where we  applied Fubini's theorem and used the identity
			\begin{equation*}
				\int_{\max\set{r_0,r_1}}^T \frac{t}{\sqrt{t^2-r_0^2}\sqrt{t^2-r_1^2}}dt=\log\left(\sqrt{T^2-r_0^2}+\sqrt{T^2-r_1^2}\right)-\log{\abs{r_1^2-r_0^2}} \,.
			\end{equation*}
			Next, we use integration by parts and Fubini's theorem again on \eqref{eq:integral1} to obtain
			\begin{align*}
				-\int_0^Tf(x)r_1&\M g(x,r_1)\log\left(T+\sqrt{T^2-r_1^2}\right)dr_1\\
				&+\int_0^T\int_0^T\frac{r_0r_1\M f(x,r_0)\M g(x,r_1)}{\left(\sqrt{T^2-r_0^2}+\sqrt{T^2-r_1^2}\right)\sqrt{T^2-r_0^2}}dr_0 dr_1
			\end{align*}
			where the first integral can be extended to
			\begin{align*}
				-\int_0^Tf(x)&\left(\int_{r_1}^T\frac{r_1\M g(x,r_1)}{\sqrt{t^2-r_1^2}}dt+\log(r_1)r_1\M g(x,r_1)\right)dr_1\\
				&=-\int_0^Tf(x)v(x,t)dt-\int_0^T f(x)\log(r_1)r_1\M g(x,r_1) dr_1.
			\end{align*}
			by a further application of Fubini's theorem. Thus, we finally have
			\begin{align*}
				\int_0^T u(x,t)v(x,t)dt&=\int_0^T f(x)\log(r_1)r_1\M g(x,r_1) dr_1\\
				&+\int_0^T\int_0^T \frac{r_0r_1\M f(x,r_0)\M g(x,r_1)}{\left(\sqrt{T^2-r_0^2}+\sqrt{T^2-r_1^2}\right)\sqrt{T^2-r_0^2}}dr_0 dr_1\\
				&-\frac{1}{2}\int_0^T\int_0^T \partial_r\M f(x,r_0)r_1\M g(x,r_1)\log{\abs{r_1^2-r_0^2}}dr_1dr_0.
			\end{align*}
			Letting $T\to \infty$ and integrating both sides afterwards we see that \eqref{eq:identity1} holds.
			
			In the last two steps we reshape both integrals in \eqref{eq:identity1} on the right side to prove the final statement.
			\item \label{item:step2}Using the definition of spherical mean operator and polar coordinates we observe that second integral on in the right hand side of \eqref{eq:identity1} can be evaluated to
			\begin{multline*}
				-\frac{1}{2}\int_\Omega\int_0^\infty\int_0^\infty\partial_r\M f(x,r_0)r_1\M g(x,r_1)\log(\abs{r_1^2-r_0^2})dr_0 dr_1 dx
				\\ = -\frac{1}{4\pi}\int_\Omega\int_0^\infty\int_{\R^2}\partial_r\M f(x,r_0)g(x+y)\log\left(\abs{r_0^2-\norm{y}^2}\right)dy dr_0 dx.
			\end{multline*}
			One further application of Fubini's theorem and substitution of $y$ with $y-x$ lead then to
			\begin{equation*}
				-\frac{1}{4\pi}\int_{\R^2}g(x)\int_{\R^2}\chi_\Omega(y)\int_0^\infty \partial_r\M f(x,r_0)\log\left(\abs{r_0^2-\norm{x-y}^2}\right)dr_0 dy dx
			\end{equation*}
			and hence, applying Lemma \ref{lem:lemidentity1}, we see that this integral coincides with
			\begin{equation}
			\label{eq:identity2}
			\begin{aligned}
				-\frac{1}{8\pi^2}\int_\Omega\int_\Omega f(x)g(y)&\frac{\left(\Hilbert_s\Radon \chi_\Omega\right)\left(\tilde{n}(x,y),\tilde{s}(x,y)\right)}{\norm{x-y}}dxdy\\
				&+\frac{1}{4\pi}\int_\Omega g(x)\int_\Omega f(y)\log\left(\norm{x-y}^2\right)dy dx.
			\end{aligned}
			\end{equation}
			\item Finally, using polar coordinates and applying the substitution rule for the first integral in \eqref{eq:identity1} yields
			\begin{equation*}
				-\int_\Omega\int_0^\infty f(x)\log(r_1)r_1\M g(x,r_1)dr_1 dx=\frac{1}{4\pi}\int_\Omega f(x)\int_\Omega g(y)\log\left(\norm{x-y}^2\right)dy dx \,.
			\end{equation*}
		Changing the order of integration in the last displayed equation and  combining this with Items \ref{item:step1} and \ref{item:step2} show the claimed statement.\qedhere
		\end{enumerate}
	\end{proof}

\section{Main results}
		
In this section  we present and prove our inversion formulas for the inversion of the two-dimensional wave equation from Neumann measurements.

\subsection{Formula for Neumann traces on convex domains}\label{subsec:explformconv}

The first main result is an explicit reconstruction integral  that applies  to
arbitrary convex domains and yields an  exact inversion formula
from the Neumann trace  up to an explicitly given smoothing integral operator. 	
	
	\begin{theorem}
		\label{thm:explicitform}
		Let $\Omega\subset \R^2$ be a bounded convex domain with smooth boundary, $f\in C_c^\infty(\Omega)$ and  $u\colon \R^2\times[0,\infty)\to\R$ be the solution of the wave equation \eqref{eq:waveeq} with initial data $(f,0)$
		Then, for every $x\in\Omega$, we have
		\begin{equation}
			\label{eq:explictform}
			f(x)=\frac{1}{\pi}\int_{\partial\Omega}\int_{\norm{x-y}}^\infty \frac{\nd u(y,t)}{\sqrt{t^2-\norm{x-y}^2}}dtd\sigma(y)+\K_\Omega f(x)\,,
		\end{equation}
		where the integral operator $\K_\Omega$ is defined by
		\eqref{eq:K}.
\end{theorem}

We will derive Theorem \ref{thm:explicitform} from the following Proposition which is the key ingredient for the derivation of all main results in this paper.
	\begin{prop}
		\label{prop:thmidentity3}
		Let $f,g\in C_c^\infty(\Omega)$. Then the following identity holds:
		\begin{equation}
			\int_\Omega f(x)g(x) dx=2\int_{\partial\Omega}\int_0^\infty v(x,t)\nd u(x,t) dt d\sigma(x)+\int_\Omega (\K_\Omega f)(x)g(x)dx,
		\end{equation}
		where $u$ is the solution of \eqref{eq:waveeq} with initial data $(f,0)$ and $v$ the solution of \eqref{eq:waveeq2} with initial data $(0,g)$.
	\end{prop}
	
	\begin{proof}
		\begin{enumerate}[wide=\parindent,label=(\roman*)]
			\item We first show that
			\begin{equation}
				\label{eq:identity3}
				\int_\Omega f(x)g(x) dx=\int_0^\infty \int_\Omega v(x,t)\Delta u(x,t)-u(x,t)\Delta v(x,t)dxdt.
			\end{equation}
			Application of integration by parts on the two inner integrals yields
			\begin{align*}
				&\int_\Omega\int_0^\infty v(x,t)\Delta u(x,t)dtdx=-\int_\Omega\int_0^\infty \partial_t u(x,t)\partial_t v(x,t)dtdx\quad \text{and}\\
				&\int_\Omega\int_0^\infty u(x,t)\Delta v(x,t)dtdx=-\int_\Omega f(x)g(x)dx-\int_\Omega\int_0^\infty \partial_t u(x,t)\partial_t v(x,t)dtdx.
			\end{align*}
			Therefore, the subtraction of both integrals and changing order of integration lead to the desired result.
			\item In the next step we use Green's second identity on right inner integral in \eqref{eq:identity3} to obtain
			\begin{equation*}
				\int_\Omega v(x,t)\Delta u(x,t)-u(x,t)\Delta v(x,t)dx=\int_{\partial\Omega} v(x,t)\nd u(x,t)-u(x,t)\nd v(x,t) d\sigma(x),
			\end{equation*}
			where the right integral can be written as \[\int_{\partial\Omega}\scp{v(x,t)\nabla u(x,t)-u(x,t)\nabla v(x,t),\nu(x)}d\sigma(x).\] Then the product rule yields the relation
			\begin{align*}
				\int_{\partial\Omega}&\scp{v(x,t)\nabla u(x,t)-u(x,t)\nabla v(x,t),\nu(x)}d\sigma(x)\\
				&=\int_{\partial\Omega}\scp{2v(x,t)\nabla u(x,t)-\left(\nabla uv \right)(x,t),\nu(x)}d\sigma(x)\\
				&=2\int_{\partial\Omega}v(x,t)\nd u(x,t)d\sigma(x)-\int_{\partial\Omega}\scp{\left(\nabla uv\right)(x,t),\nu(x)}d\sigma(x),
			\end{align*}
			From the divergence theorem we conclude that the second integral on the right side can be evaluated to
			\begin{align*}
				\int_{\partial\Omega}\scp{\left(\nabla uv\right)(x,t),\nu(x)}d\sigma(x)&=\int_\Omega \div \left(\nabla uv\right)(x,t) dx\\
				&=\int_\Omega \left(\Delta uv\right)(x,t)dx.
			\end{align*}
			\item Now, it remains to show that
			\begin{equation*}
				\int_\Omega \int_0^\infty \left(\Delta uv\right)(x,t)dtdx=-\int_\Omega (\K_\Omega f)(x)g(x)dx.
			\end{equation*}
			First, we notice that one can easily verify the relation $\Delta uv = v\Delta u+2\scp{\nabla u,\nabla v}+u\Delta v$. Furthermore, from \eqref{eq:waveeq} and \eqref{eq:waveeq2} we observe that $\Delta u$  and $\Delta v$ are the solutions of \eqref{eq:waveeq} and \eqref{eq:waveeq2} with respect to $\Delta f$ and $\Delta g$. The same also holds for the gradients $\nabla u$ and $\nabla v$. This observation and Lemma \ref{lem:lemidentity2} imply then
			\begin{equation}
			\label{eq:identity4}
			\begin{aligned}
			\int_\Omega &\int_0^\infty \left(\Delta uv\right)(x,t)dtdx\\
			&=-\frac{1}{8\pi^2}\int_\Omega\int_\Omega (\nabla_x + \nabla_y)^2\left(f(x)g(y)\right)\frac{\left(\Hilbert_s\Radon \chi_\Omega\right)\left(\tilde{n}(x,y),\tilde{s}(x,y)\right)}{\norm{x-y}}dxdy,
			\end{aligned}
			\end{equation}
			where \[(\nabla_x + \nabla_y)^2f(x)g(y)=g(y)\Delta f(x)+2\scp{\nabla f(x),\nabla g(y)}+f(x) \Delta g(y).\] Next, we apply Green's first identity and integration by parts on the right integral in \eqref{eq:identity4} to deduce
			\begin{align*}
				\int_\Omega &\int_0^\infty \left(\Delta uv\right)(x,t)dtdx\\
			&=\frac{1}{8\pi^2}\int_\Omega\int_\Omega f(x)g(y)(\nabla_x + \nabla_y)^2\frac{\left(\Hilbert_s\Radon \chi_\Omega\right)\left(\tilde{n}(x,y),\tilde{s}(x,y)\right)}{\norm{x-y}}dxdy.
			\end{align*}
			Thus, we are left to show \[(\nabla_x + \nabla_y)^2\frac{\left(\Hilbert_s\Radon \chi_\Omega\right)\left(\tilde{n}(x,y),\tilde{s}(x,y)\right)}{\norm{x-y}}=\frac{\left(\partial_s^2\Hilbert_s\Radon \chi_\Omega\right)\left(\tilde{n}(x,y),\tilde{s}(x,y)\right)}{\norm{x-y}}.\] This is straightforward computation: Applying the chain rule on the numerator yields
			\begin{equation*}
				(\nabla_x + \nabla_y)\left(\Hilbert_s\Radon \chi_\Omega\right)\left(\tilde{n}(x,y),\tilde{s}(x,y)\right)=\left(\partial_s^2\Hilbert_s\Radon \chi_\Omega\right)\left(\tilde{n}(x,y),\tilde{s}(x,y)\right)\frac{y-x}{\norm{x-y}}.
			\end{equation*}
			Then, from product rule and the relation $(\nabla_x + \nabla_y)\norm{x-y}^{-1} = 0$ we conclude
			\begin{equation*}
				(\nabla_x + \nabla_y)\frac{\left(\Hilbert_s\Radon \chi_\Omega\right)\left(\tilde{n}(x,y),\tilde{s}(x,y)\right)}{\norm{x-y}}=\left(\Hilbert_s\Radon \chi_\Omega\right)\left(\tilde{n}(x,y),\tilde{s}(x,y)\right)\frac{y-x}{\norm{x-y}^2}.
			\end{equation*}
			One further application of the chain rule and product rule finally yields then
			\begin{align*}
				(\nabla_x + \nabla_y)^2&\frac{\left(\Hilbert_s\Radon \chi_\Omega\right)\left(\tilde{n}(x,y),\tilde{s}(x,y)\right)}{\norm{x-y}}\\
				&=(\partial_{x_1}+\partial_{y_1})\left(\Hilbert_s\Radon \chi_\Omega\right)\left(\tilde{n}(x,y),\tilde{s}(x,y)\right)\frac{y_1-x_1}{\norm{x-y}^2}\\
				&\quad+(\partial_{x_2}+\partial_{y_2})\left(\Hilbert_s\Radon \chi_\Omega\right)\left(\tilde{n}(x,y),\tilde{s}(x,y)\right)\frac{y_2-x_2}{\norm{x-y}^2}\\
				&=\frac{\left(\partial_s^2\Hilbert_s\Radon \chi_\Omega\right)\left(\tilde{n}(x,y),\tilde{s}(x,y)\right)}{\norm{x-y}}.\qedhere
			\end{align*}
		\end{enumerate}
	\end{proof}

Now we can proof our main result. 	

\begin{proof}[Proof of Theorem \ref{thm:explicitform}]
		First, let $g\in C_c^\infty(\Omega)$ and $v\colon \R^2\times[0,\infty)\to\R$ denote the solution of \eqref{eq:waveeq2} with initial data $(0,g)$. As in Lemma \ref{eq:identity3} stated, we have that
		\begin{equation}
			\label{eq:identity7}
			\int_\Omega f(x)g(x) dx=2\int_{\partial\Omega}\int_0^\infty v(x,t)\nd u(x,t) dt d\sigma(x)+ \int_\Omega (\K_\Omega f)(x)g(x)dx.
		\end{equation}
		By using polar coordinates, Equation~\eqref{eq:wavesol2} and substitution rule we have
		\begin{align*}
			v(x,t)&=\frac{1}{2\pi}\int_0^\infty \int_{\Sp^1} \frac{\chi_{(0,t)}(r)}{\sqrt{t^2-r^2}}g(x+r\omega)d\sigma(\omega) dr\\
			&=\frac{1}{2\pi}\int_{\R^2}\frac{\chi_{(0,t)}(\norm{y})}{\sqrt{t^2-\norm{y}^2}}g(x+y)dy\\
			&=\frac{1}{2\pi}\int_{\R^2}\frac{\chi_{(0,t)}(\norm{x-y})}{\sqrt{t^2-\norm{x-y}^2}}g(y)dy.
		\end{align*}
		Thus, inserting this relation into the second integral in \eqref{eq:identity7} and applying Fubini's theorem yields then
		\begin{align*}
			2\int_{\partial\Omega}\int_0^\infty &v(x,t)\nd u(x,t) dt d\sigma(x)\\
			&=\frac{1}{\pi}\int_{\R^2}g(y)\int_{\partial\Omega}\int_0^\infty\frac{\chi_{(0,t)}(\norm{x-y})}{\sqrt{t^2-\norm{x-y}^2}} \nd u(x,t)dt d\sigma(x) dy\\
			&=\frac{1}{\pi}\int_\Omega g(x)\int_{\partial\Omega}\int_{\norm{x-y}}^\infty \frac{\nd u(y,t)}{\sqrt{t^2-\norm{x-y}^2}}dtd\sigma(y) dx.
		\end{align*}
		This leads finally to
		\begin{equation*}
			\int_\Omega f(x)g(x) dx=\int_{\Omega}\left(\int_{\partial\Omega}\int_{\norm{x-y}}^\infty \frac{\nd u(y,t)}{\sqrt{t^2-\norm{x-y}^2}}dtd\sigma(y)+(\K_\Omega f)(x)\right)g(x)dx.
		\end{equation*}
		Since this identity holds for every test function $g\in C_c^\infty(\Omega)$, the claimed inversion formula \eqref{eq:explictform} holds.
	\end{proof}

	\subsection{Exact formula for Neumann traces on ellipses}
	\label{subsec:exactform}
	
	Next, we present a formula that exactly recovers the initial data $f\colon \R^2\to\R$ from Neumann measurements given in 	
	\eqref{eq:neumannmeas}  in the case that $\Omega$
	is bounded by an ellipse
\begin{equation*}
\partial \Omega = Q \left( \set{ (x_1,x_2) \in \R^2
\Bigm|   \frac{x_1^2}{e_1^2}
+ \frac{x_2^2}{e_2^2} = 1} \right)
\end{equation*}
for numbers   $ e_1, e_ 2 >  0$ and an orthogonal
transform $Q \in \R^{2\times 2}$.

	\begin{theorem}
		\label{thm:exactform}
		Let $\Omega\subset \R^2$ be a circular or elliptical domain and  $f\in C_c^\infty(\Omega)$. Then, for every $x\in\Omega$ we have
		\begin{equation}
			\label{eq:exactformnm}
			f(x)=\frac{1}{\pi}\int_{\partial\Omega}\int_{\norm{x-y}}^\infty \frac{\nd u(y,t)}{\sqrt{t^2-\norm{x-y}^2}}dtd\sigma(y),
		\end{equation}
		where $u\colon \R^2\times[0,\infty)\to\R$ denotes the solution of the wave equation \eqref{eq:waveeq} with initial data $(f,0)$.
	\end{theorem}
	\begin{proof}
		By Theorem \ref{thm:explicitform}, we are left to show that $\K_\Omega(x)=0$ for every $x\in\Omega$. For the proof of this identity, we refer to \cite{Hal13,Hal14}.
	\end{proof}
	
	\subsection{Exact formula for mixed traces on circular domains}
		\label{subsec:exactformneumann}

Finally, we show that the initial data $f\colon \R^2\to\R$ in \eqref{eq:waveeq} can be recovered by any linear combination of the solution of the wave equation $u\colon \R^2\times[0,\infty)\to\R$ and the normal derivative $\nd u$ on circular domains.
	\begin{theorem}\label{thm:neumann}
		Let $\B_\rho^2(z)\subset \R^2$ be an open ball with radius $\rho>0$ and center $z\in\R^2$ in the plane, let $u$ denote the solution of the wave equation \eqref{eq:waveeq} with initial data $(f,0)$, where $f\in C_c^\infty(\B_\rho^2(z))$, and $a \geq 0$ and $b >0$.
		Then, for every $x\in\B_\rho^2(z)$, we have,
		\begin{equation}
			\label{eq:exactformmixed}
			f(x)=\frac{1}{b\pi}\int_{\partial \B_\rho^2(z)}\int_{\norm{x-y}}^\infty \frac{au(y,t)+b\nd u(y,t)}{\sqrt{t^2-\norm{x-y}^2}}dtd\sigma(y) \,.
		\end{equation}
	\end{theorem}

The proof of Theorem~\eqref{thm:neumann} uses the following
lemma, which  is a range condition for the Dirichlet trace on a circle
of the wave equation.

	\begin{lemma}
		\label{lem:lemidentity3}
		Let $f\in C_c^\infty(\B_\rho^2(z))$, where $\rho>0$ and $z\in\R^2$, and $u\colon \R^2\times[0,\infty)\to\R$ the solution of wave equation \eqref{eq:waveeq} with initial data $(f,0)$. Then for every $x\in\B_\rho^2(z)$ we have
		\begin{equation*}
			0=\int_{\partial\B_\rho^2(z)}\int_{\norm{x-y}}^\infty \frac{u(y,t)}{\sqrt{t^2-\norm{x-y}^2}}dtd\sigma(y).
		\end{equation*}

	\end{lemma}
	\begin{proof}
It is sufficient  to consider case where $\B_\rho^2(z) =  \B^2$
is the unit ball centered at the origin. The general case follows from translation and rescaling. Suppose $f\in C_c^\infty(\B^2)$ has compact support in the open unit ball $\B^2$. For  this case, in \cite{FinHalRak07} the identities
		\begin{align}
			f(x)&=-\frac{1}{\pi}\int_{\Sp^1}\int_{\norm{x-y}}^\infty \frac{t \partial_t^2 v(y,t)}{\sqrt{t^2-\norm{x-y}^2}}dtd\sigma(y)\label{eq:identity5}\\
			f(x)&=-\frac{1}{\pi}\int_{\Sp^1}\int_{\norm{x-y}}^\infty \frac{\partial_t (t\partial_t v)(y,t)}{\sqrt{t^2-\norm{x-y}^2}}dtd\sigma(y) \,,\label{eq:identity6}
		\end{align}
		for any $x\in \B^2$, have been shown, where $v\colon \R^2\times[0,\infty)\to\R$ is the solution of the wave equation \eqref{eq:waveeq2} with initial data $(0,f)$. Thus, by using product rule in \eqref{eq:identity6} and identity \eqref{eq:identity5} we have for every $x\in\B^2$
		\begin{align*}
			f(x)&=-\frac{1}{\pi}\int_{\Sp^1}\int_{\norm{x-y}}^\infty \frac{\partial_t v(y,t)+t\partial_t^2 v(y,t)}{\sqrt{t^2-\norm{x-y}^2}}dtd\sigma(y)\\
			&=-\frac{1}{\pi}\int_{\Sp^1}\int_{\norm{x-y}}^\infty \frac{\partial_t v(y,t)}{\sqrt{t^2-\norm{x-y}^2}}dtd\sigma(y)+f(x) \,,
		\end{align*}
		which implies
		\begin{equation*}
			0=\int_{\Sp^1}\int_{\norm{x-y}}^\infty \frac{\partial_t v(y,t)}{\sqrt{t^2-\norm{x-y}^2}}dtd\sigma(y) \,.
		\end{equation*}
		From \eqref{eq:wavesol1}, \eqref{eq:wavesol2} we also have the identity  $\partial_t v=u$. This shows the claimed identity.
	\end{proof}

We are now ready to proof the inversion formula for
mixed trace as a corollary of the inversion formula
for the Neumann trace  in Theorem \ref{thm:exactform}
and the range condition for the Dirichlet trace
derived in     Lemma \ref{lem:lemidentity3}.

	\begin{proof}[Proof of Theorem \ref{thm:neumann}]
The following holds:
		\begin{itemize}
		\item
		From Lemma \ref{lem:lemidentity3} we have
		for every $x\in\B_\rho^2(z)$
		\begin{equation*}
			0=\frac{1}{b\pi}\int_{\partial\B_\rho^2(z)}\int_{\norm{x-y}}^\infty \frac{au(y,t)}{\sqrt{t^2-\norm{x-y}^2}}dtd\sigma(y) \,.
		\end{equation*}
		\item From Theorem \ref{thm:exactform} we obtain for every $x\in\B_\rho^2(z)$
		\begin{equation*}
			f(x)=\frac{1}{b\pi}\int_{\partial\B_\rho^2(z)}\int_{\norm{x-y}}^\infty \frac{b\nd u(y,t)}{\sqrt{t^2-\norm{x-y}^2}}dtd\sigma(y) \,.
		\end{equation*}
		\end{itemize}
Adding the last two displayed  equations  leads  the desired result.
	\end{proof}
	
	\section{Numerical experiments}
	\label{sec:num}
	
	In this section, we give details of the numerical implementation of the derived exact inversion formula \eqref{eq:exactformnm} and show numerical results. The inversion formula \eqref{eq:exactformmixed} with mixed measurements can be implemented in the same way. Throughout this  section, we assume a circular domain $\B_\rho^2(z)$ with radius $\rho>0$ and center $z\in\R^2$.
	
	\subsection{Discretization of initial data and normal derivative}
	
	Suppose that we have given discrete data
	\begin{equation*}
		\mathbf{f}[i,j]=f(x_{i,j}),\quad (i,j)\in \set{1,\ldots,N}^2
	\end{equation*}
	of the initial data $f\in C_c^\infty(\B_\rho^2(z))$, where $N$ is the image size, $\Delta x\coloneqq 2\rho/(N-1)$ the step size and $x_{i,j}\coloneqq (-\rho+z_1+(i-1)\Delta x,-\rho+z_2+(j-1)\Delta x)$.
	 Next, we assume that we have $M\coloneqq \lceil 2\rho\pi/\Delta x \rceil$ detector points located at
	\begin{equation*}
		y_k\coloneqq z+\rho(\cos\varphi_k,\sin\varphi_k)\in \partial\B_\rho^2(z),\quad k\in\set{1,\ldots,M}^2
	\end{equation*}
	where $\varphi_k\coloneqq (k-1)\frac{2\pi}{M-1}$ for $k=1,\ldots,M$. The normal derivative is then discretized by
	\begin{equation*}
		\mathbf{d}[k,l]=\scp{(\cos\varphi_k,\sin\varphi_k),\nabla \mathbf{u}[k,l]},\quad (k,l)\in\set{1,\ldots,M}\times\set{1,\ldots,L}
	\end{equation*}
	where $\mathbf{u}[k,l]$ denotes the solution of wave equation with initial data $(f,0)$ at the point $y_k$ and time $t_l\coloneqq (l-1)\Delta x$, where $L\coloneqq \lfloor T/dx\rfloor+1$ and $T\coloneqq 16\rho$. We remark that we computed the discrete gradient $\nabla \mathbf{u}[k,l]$ by solving the wave equation on the whole grid (with FFT) and using symmetric finite differences to approximate the partial derivatives.
	\subsection{Implementation of inversion formula}
	To implement the inversion formula \eqref{eq:exactformnm}
	\begin{equation*}
			f(x)=\frac{1}{\pi}\int_{\partial\B_\rho^2(z)}\int_{\norm{x-y}}^\infty \frac{\nd u(y,t)}{\sqrt{t^2-\norm{x-y}^2}}dtd\sigma(y)
	\end{equation*}
	for $x\in\B_\rho^2(z)$ we proceed as follows: First, we approximate for all $(k,l)\in\set{1,\ldots,M}\times\set{1,\ldots,L}$ the integral
	\begin{align*}
		\int_{t_l}^{t_L} \frac{\nd u(y_k,t)}{\sqrt{t^2-t_l^2}}dt&=\sum_{j=1}^{L-1}\int_{t_j}^{t_{j+1}} \frac{\nd u(y_k,t)}{\sqrt{t^2-t_l^2}}dt\\
		&\simeq \sum_{j=1}^{L-1}\frac{\nd u(y_k,t_{j+1})}{t_{j+1}}\int_{t_j}^{t_{j+1}} \frac{t}{\sqrt{t^2-t_l^2}}dt
	\end{align*}
	by product integration. The last integral can be evaluated to \[\int_{t_j}^{t_{j+1}} \frac{t}{\sqrt{t^2-t_l^2}}dt=\sqrt{t_{j+1}^2-t_l^2}-\sqrt{t_{j}^2-t_l^2},\] and hence we have
	\begin{equation*}
		\int_{t_l}^{t_L} \frac{\nd u(y_k,t)}{\sqrt{t^2-t_l^2}}dt\simeq\sum_{j=1}^{L-1}\frac{\mathbf{d}[k,j+1]}{t_{j+1}}\left(\sqrt{t_{j+1}^2-t_l^2}-\sqrt{t_{j}^2-t_l^2}\right)\eqqcolon \mathbf{A}[k,l].
	\end{equation*}
	The above inversion formula can then be approximated by
	\begin{equation*}
		\mathbf{f}[i,j]\simeq \mathbf{f_{\mathrm{rec,n}}}[i,j]\coloneqq\frac{2\rho}{M-1}\sum_{k=1}^M\mathrm{interp}(\mathbf{A}[k,\ ],\norm{x_{i,j}-y_k}),
	\end{equation*}
	where we used $\mathrm{interp}(\mathbf{A}[k,\ ],\norm{x_{i,j}-y_k})$ to denote the interpolated value in $\norm{x_{i,j}-y_k}$ of the array $\mathbf{A}[k,\ ]$. The numerical approximation of formula \eqref{eq:exactformmixed} with mixed measurements is denoted by $\mathbf{f_{\mathrm{rec,m}}}$.
	
	Now we present numerical results of the discrete data sets $\mathbf{f_{\mathrm{rec,n}}}\in\R^{N\times N}$ and $\mathbf{f_{\mathrm{rec,m}}}\in\R^{N\times N}$.
	\subsection{Numerical results}
	The numerical approximation of the inversion formula \eqref{eq:exactformnm} mentioned above has been implemented in $\textsc{Matlab}$ and was tested on the head phantom $\mathbf{f}$ presented in Figure \ref{fig:head_phantom} with image size $N=301$, where the support of $f$ is contained in the open unit ball with $\rho=1$ and $z=0$.
	\begin{figure}[h!]
		\centering
		\includegraphics[width=6.5cm]{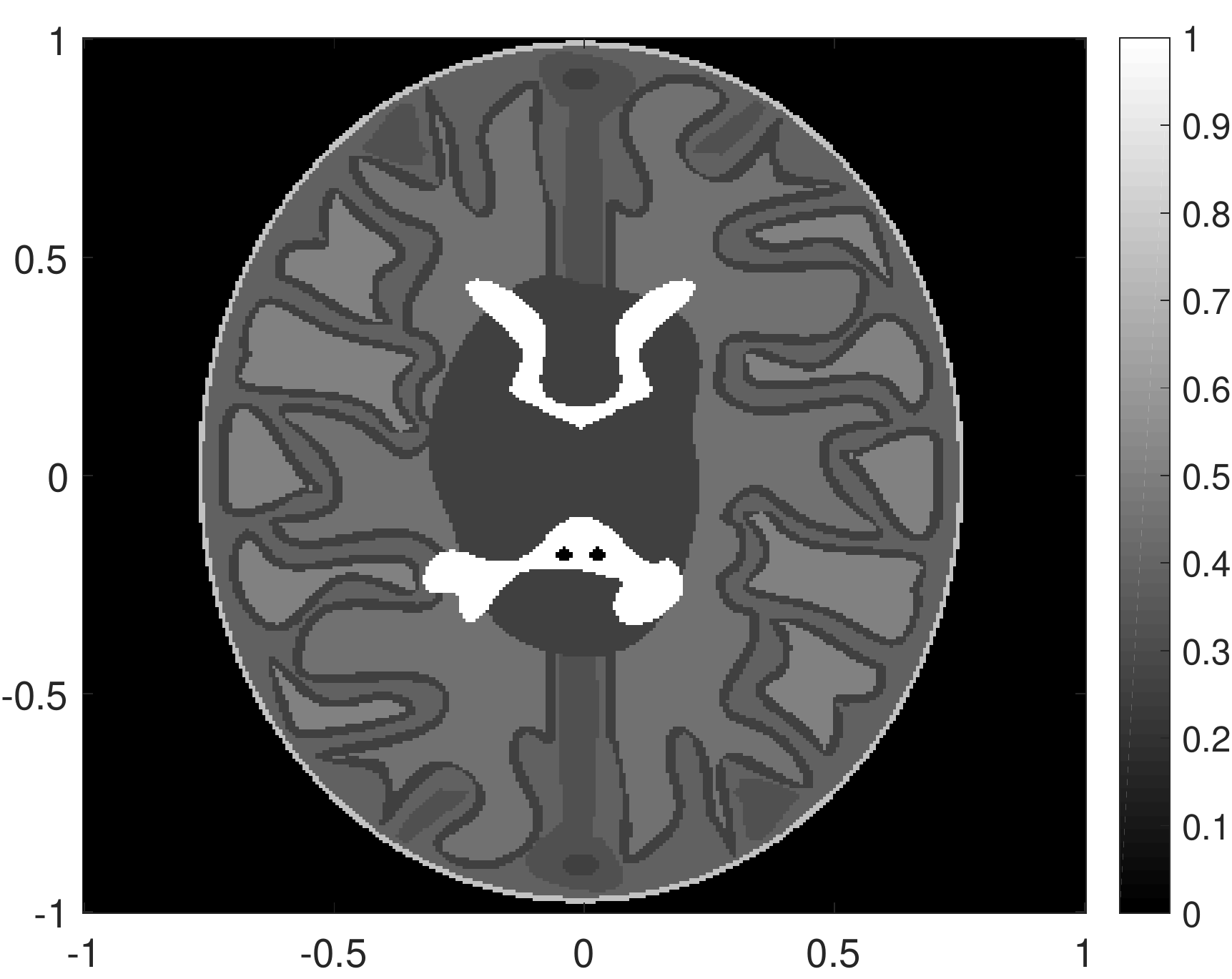}
		\caption{Self-made head phantom implemented in \textsc{Matlab} with cubic splines.}
		\label{fig:head_phantom}
	\end{figure}
	The corresponding Dirichlet and Neumann measurements of the head phantom are shown in the first column in Figure \ref{fig:data}. We additionally tested inversion formula \eqref{eq:exactformmixed} with the weights $a=1$ and $b=2\Delta x$, where the mixed measurements are also presented in the first column in Figure \ref{fig:data}. In the second column the simulated data sets with Gaussian noise (with standard deviation equal to $10\%$ of the maximal value) added are shown.
	\begin{figure}[h!]
		\centering
		\begin{tabular}{cc}
			\includegraphics[width=6.5cm]{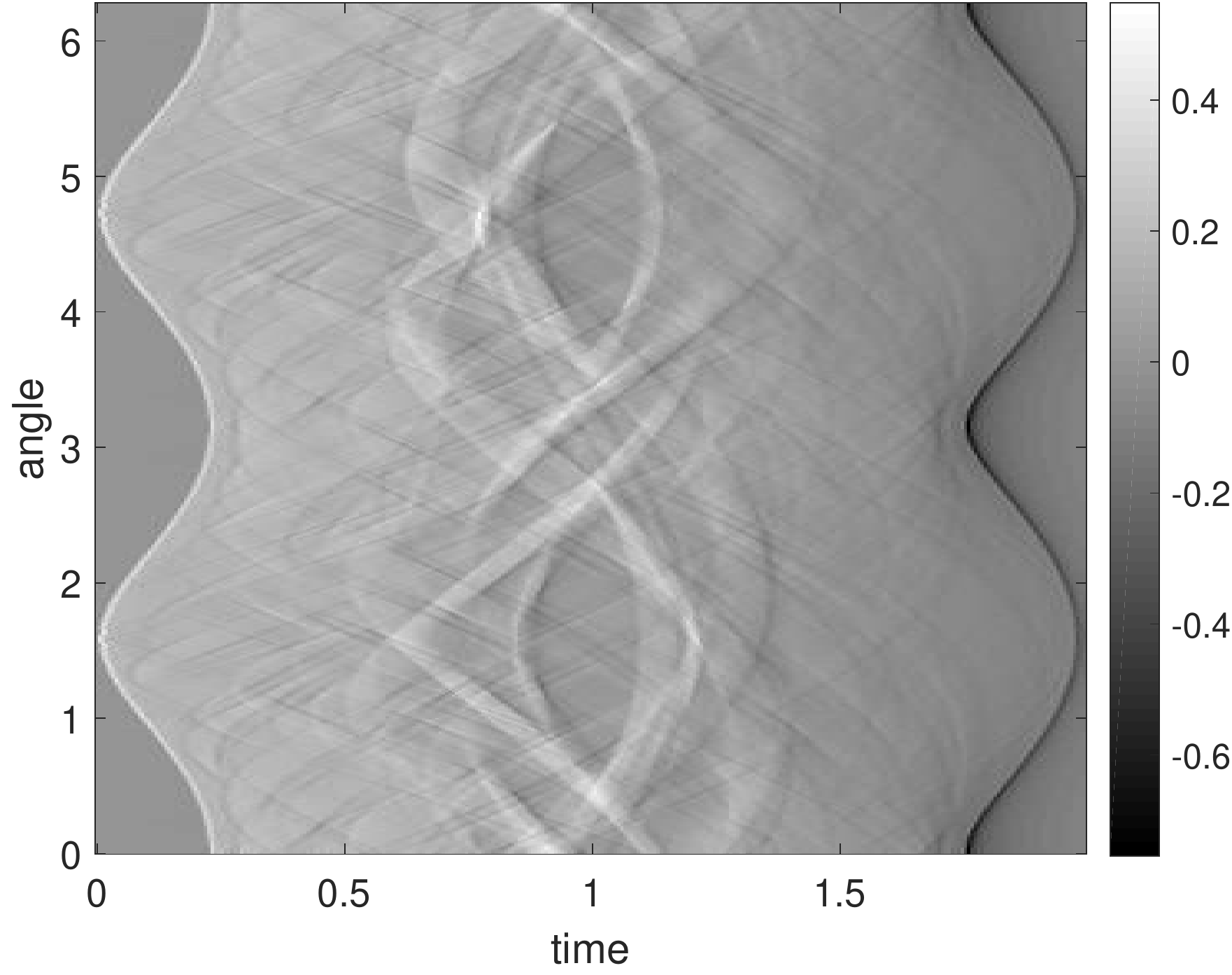}
			&
			\includegraphics[width=6.5cm]{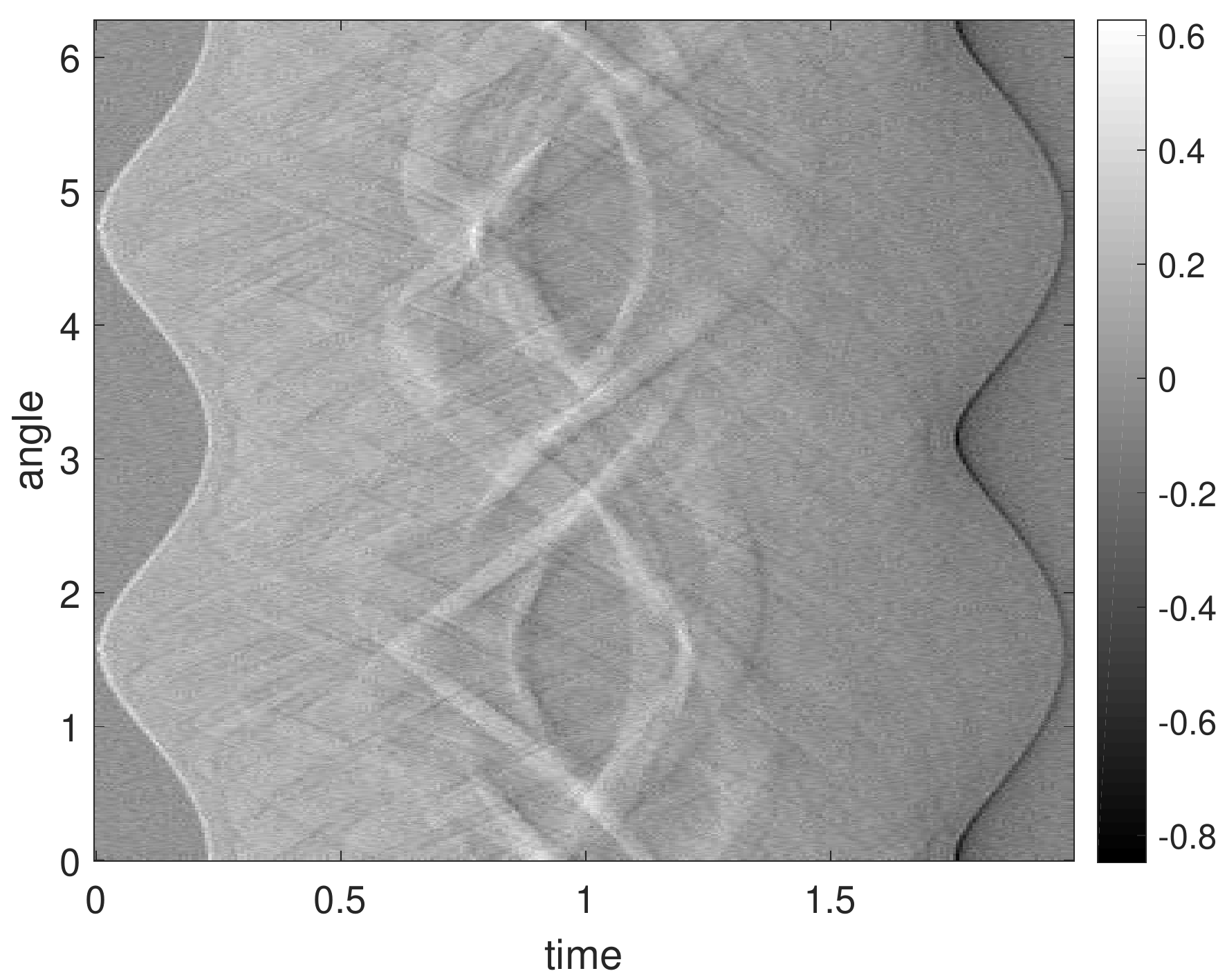} \\
			\includegraphics[width=6.5cm]{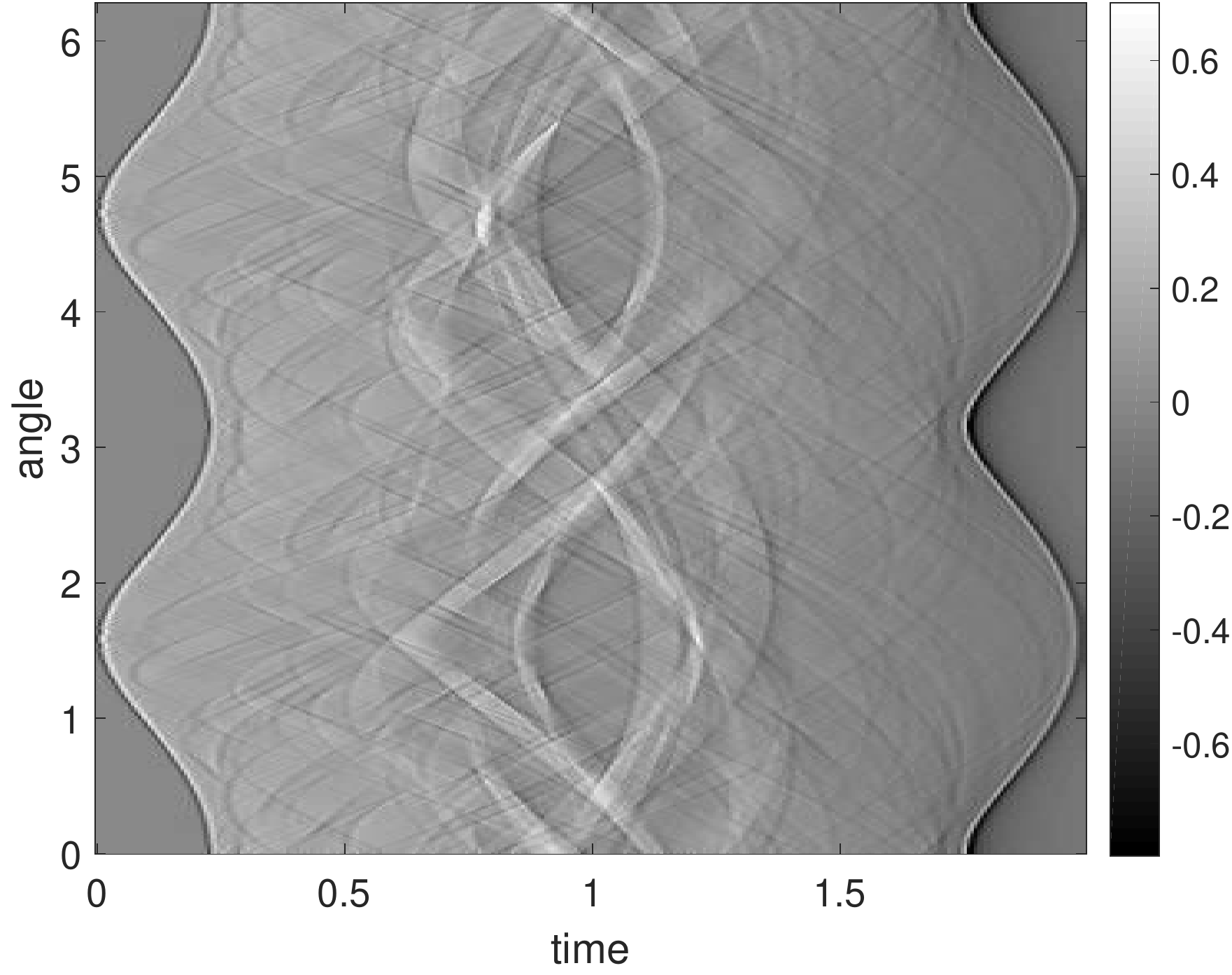}
			&
			\includegraphics[width=6.5cm]{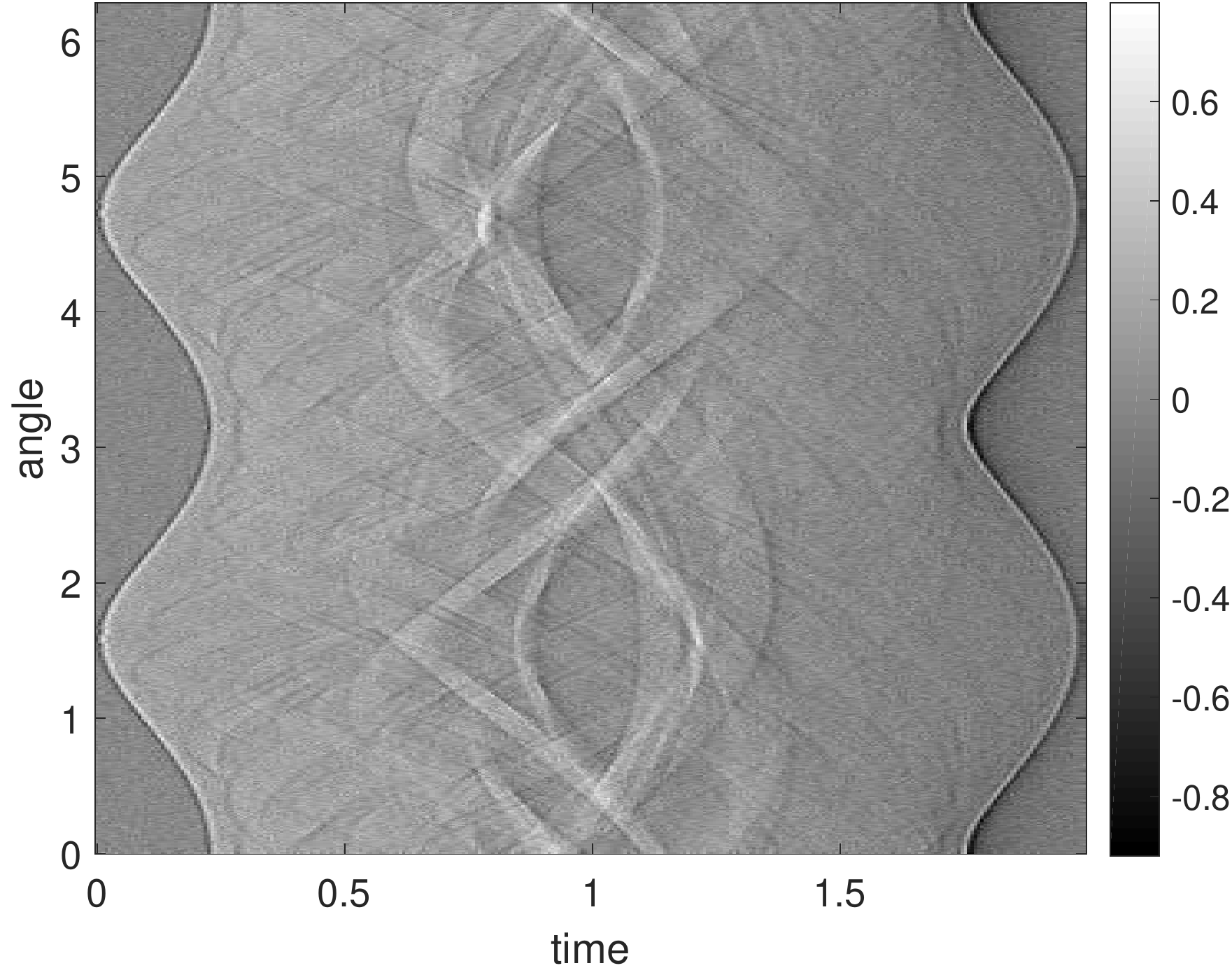} \\
			\includegraphics[width=6.5cm]{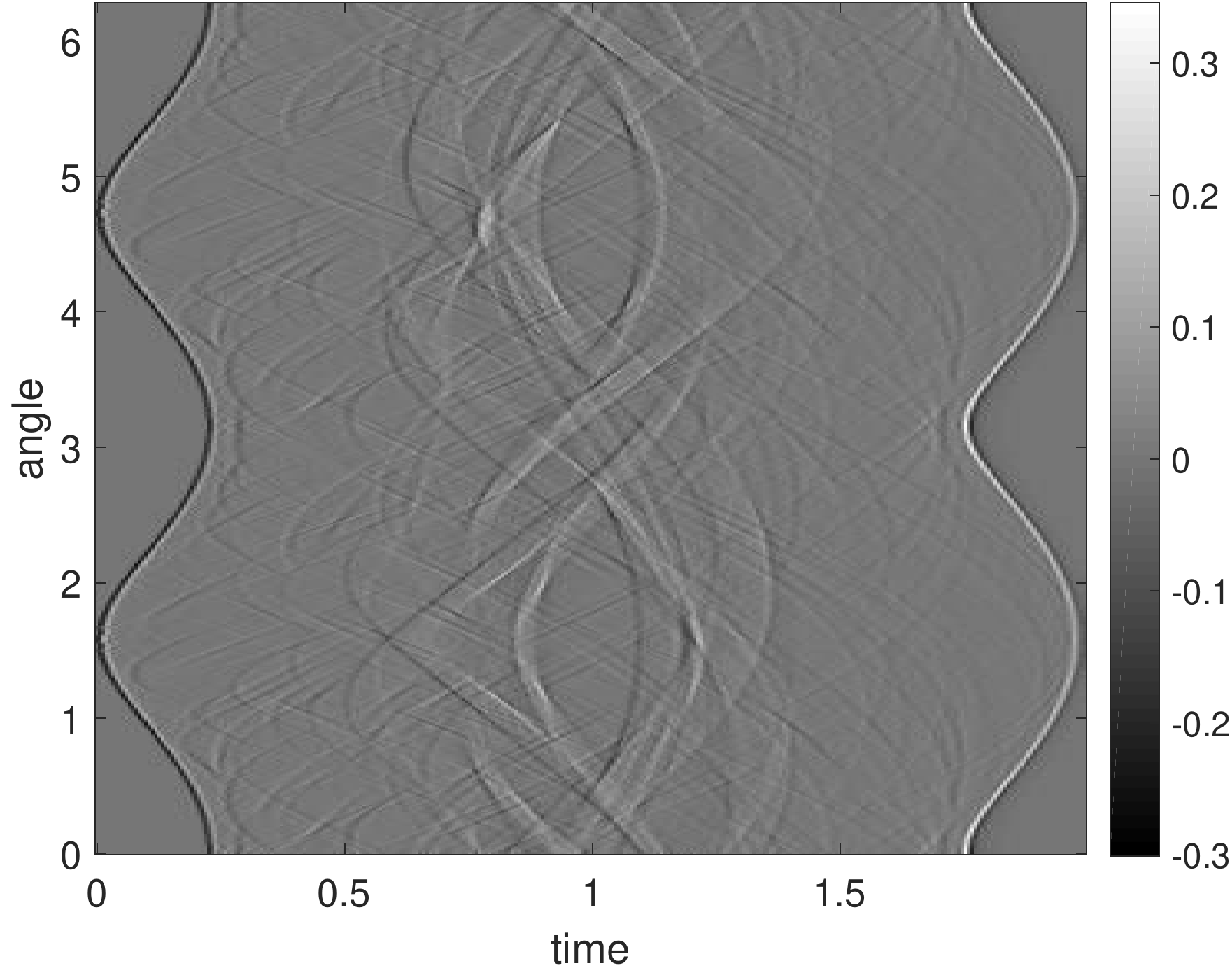}
			&
			\includegraphics[width=6.5cm]{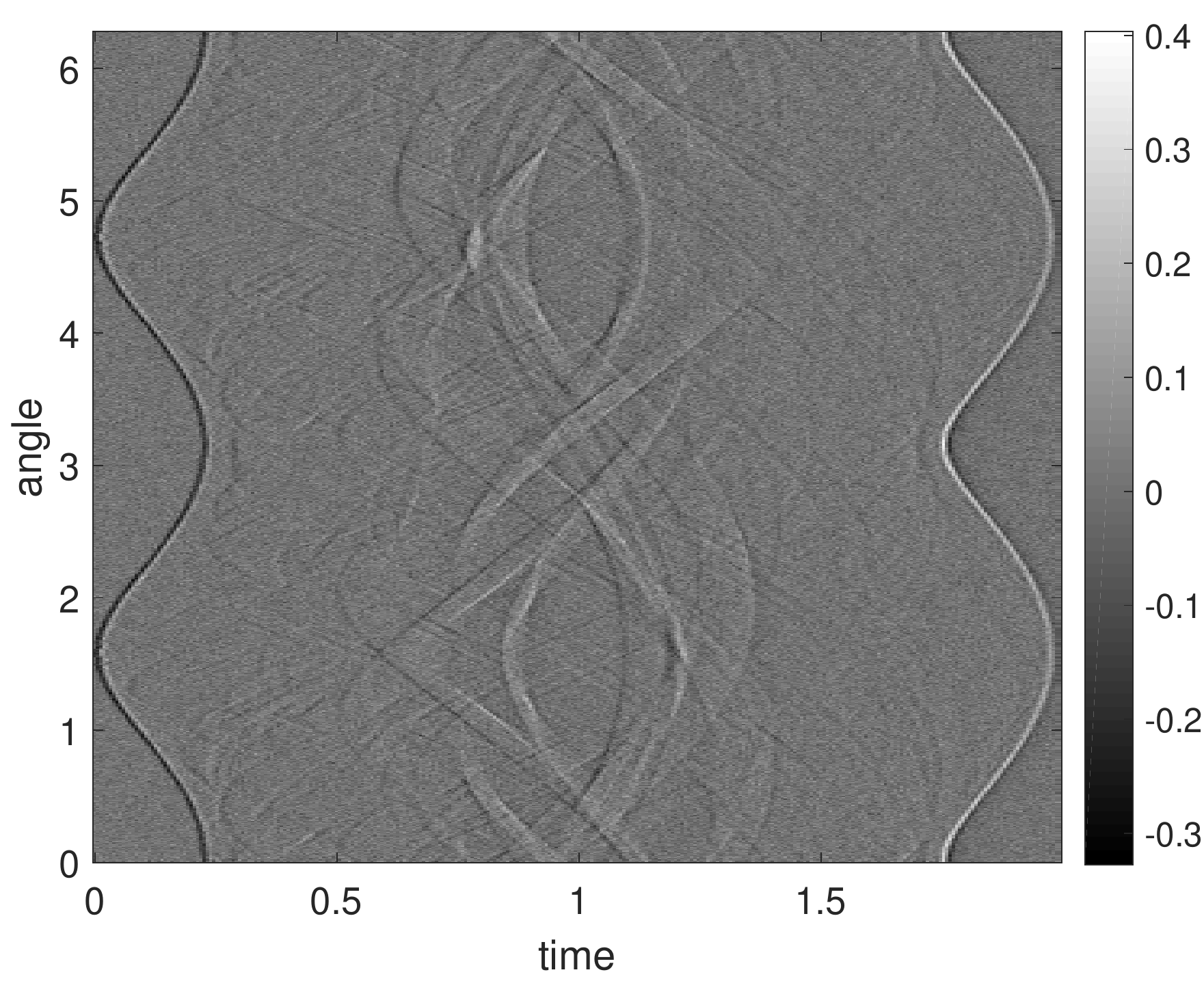}
		\end{tabular}
		\caption{Simulated data: Top, left: Dirichlet measurements $\mathbf{u}$. Top, right: Dirichlet measurements $\mathbf{u}$ with Gaussian noise added. Middle, left: mixed measurements $a\mathbf{u}+b\mathbf{d}$. Middle, right: mixed measurements $a\mathbf{u}+b\mathbf{d}$ with Gaussian noise added. Bottom, left: Neumann measurements $\mathbf{d}$. Bottom, right: Neumann measurements $\mathbf{d}$ with Gaussian noise added.}
		\label{fig:data}
	\end{figure}
	For comparison reasons, we applied the inversion formula
	\begin{equation}
		\label{eq:ubp}
		\frac{1}{a}\int_{\partial\B_\rho^2(z)}\scp{\nu(y),x-y}\int_{\norm{x-y}}^\infty \frac{(\partial_t t^{-1} au)(y,t)}{\sqrt{t^2-\norm{x-y}^2}}dtd\sigma(y),
	\end{equation}
	which exactly recovers the initial data from Dirichlet measurements (see \cite{Hal13,Hal14}), on Neumann measurements and vice versa. We denote the corresponding discrete version of \eqref{eq:ubp} by $\mathbf{f_{\mathrm{rec,d}}}$ (for implementation details, see $\mathbf{f_{\mathrm{rec,n}}}$). The first column in Figure \ref{fig:rec_exact} shows the numerical reconstructions obtained by formula \eqref{eq:ubp} and the three different simulated data sets, whereas the second column illustrates the numerical results of the inversion formulas \eqref{eq:exactformnm} and \eqref{eq:exactformmixed}, respectively. The numerical reconstructions applied to the data sets with $10\%$ Gaussian noise are shown in Figure \ref{fig:rec_noisy10}. We also tested our reconstruction method with $20\%$ Gaussian noise added to the data to observe the behaviour of the single reconstructions with a higher noise rate (see Figure \ref{fig:rec_noisy20}).
	\begin{figure}[h!]
		\centering
		\begin{tabular}{cc}
			\includegraphics[width=6.5cm]{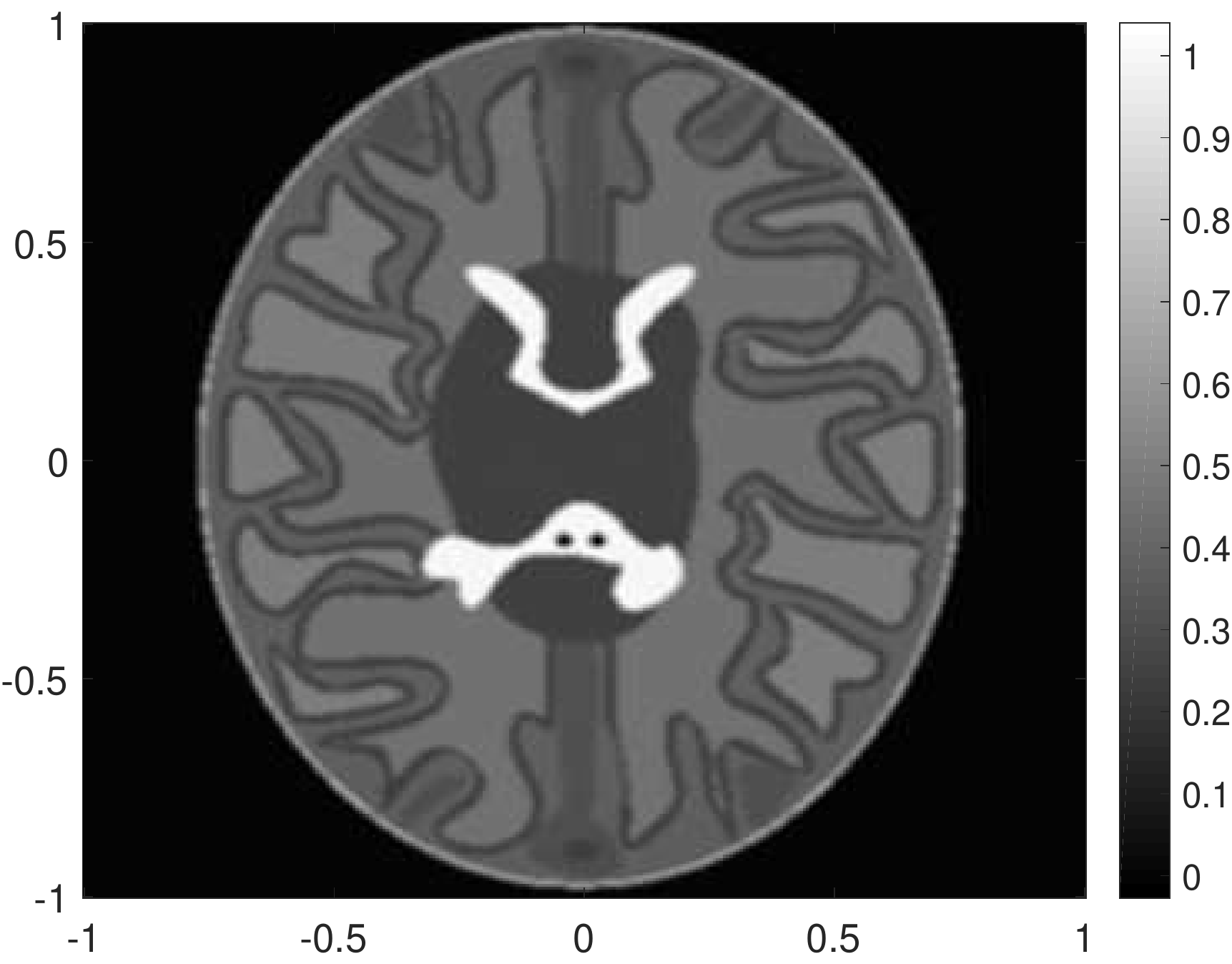}
			&
			\includegraphics[width=6.5cm]{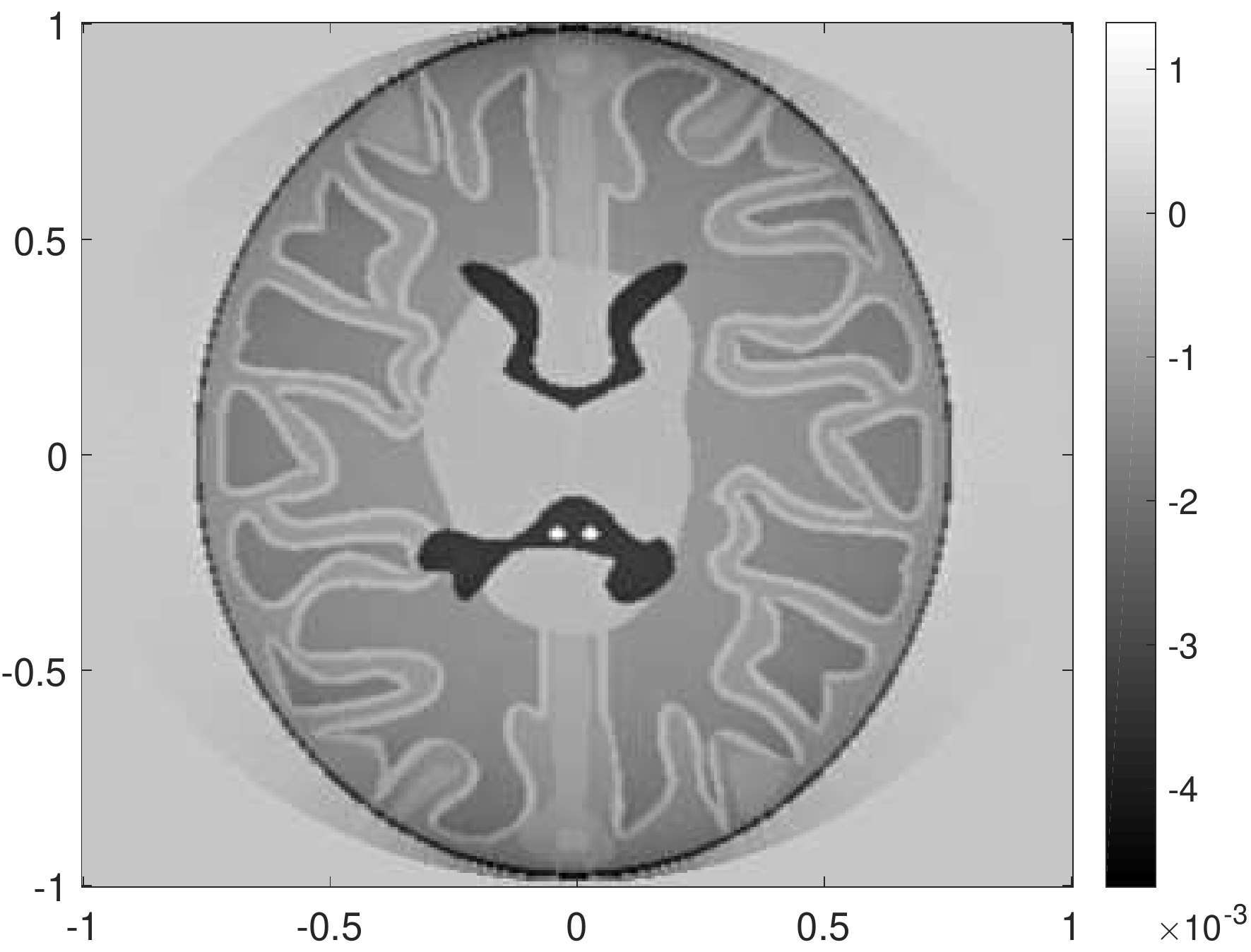} \\
			\includegraphics[width=6.5cm]{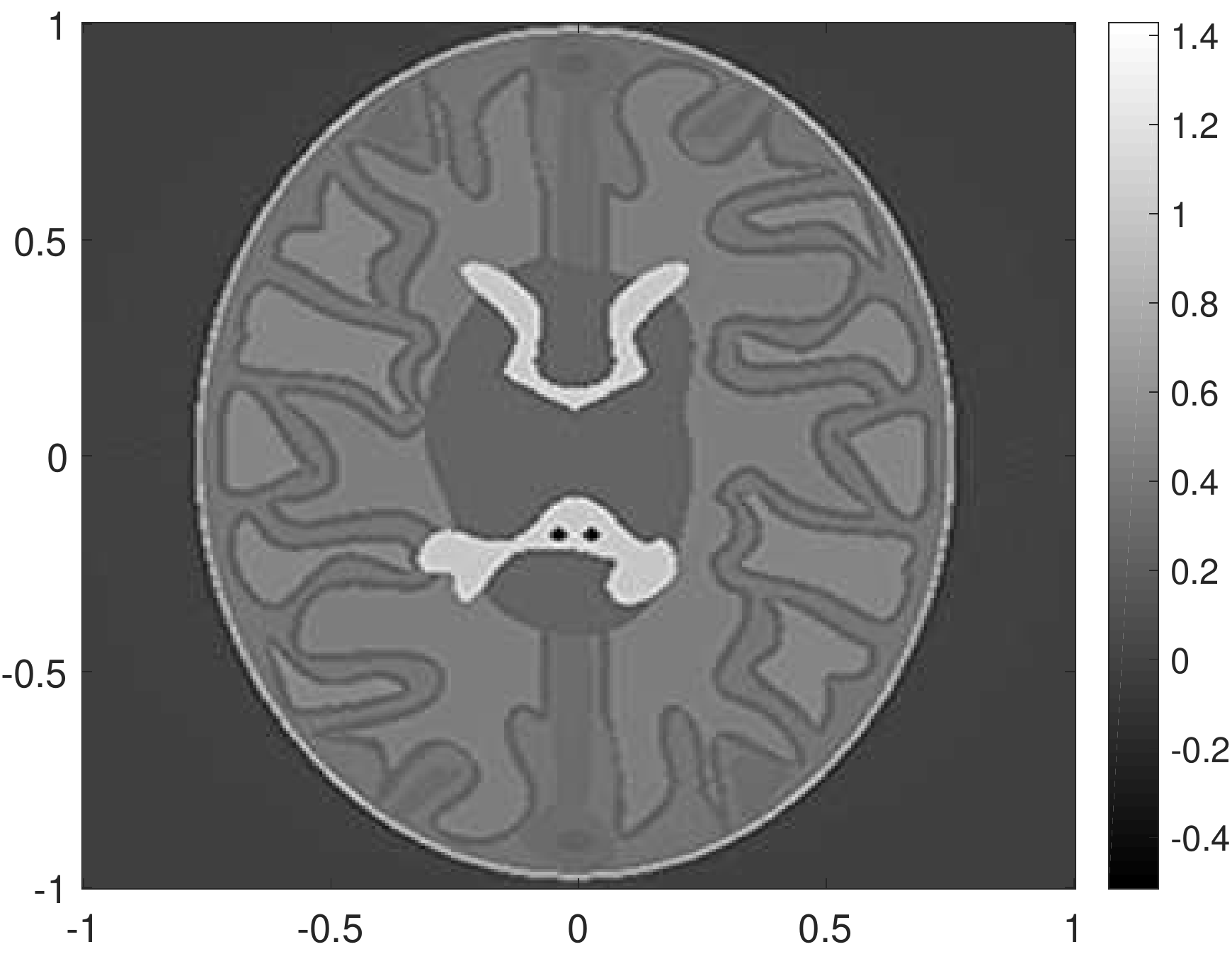}
			&
			\includegraphics[width=6.5cm]{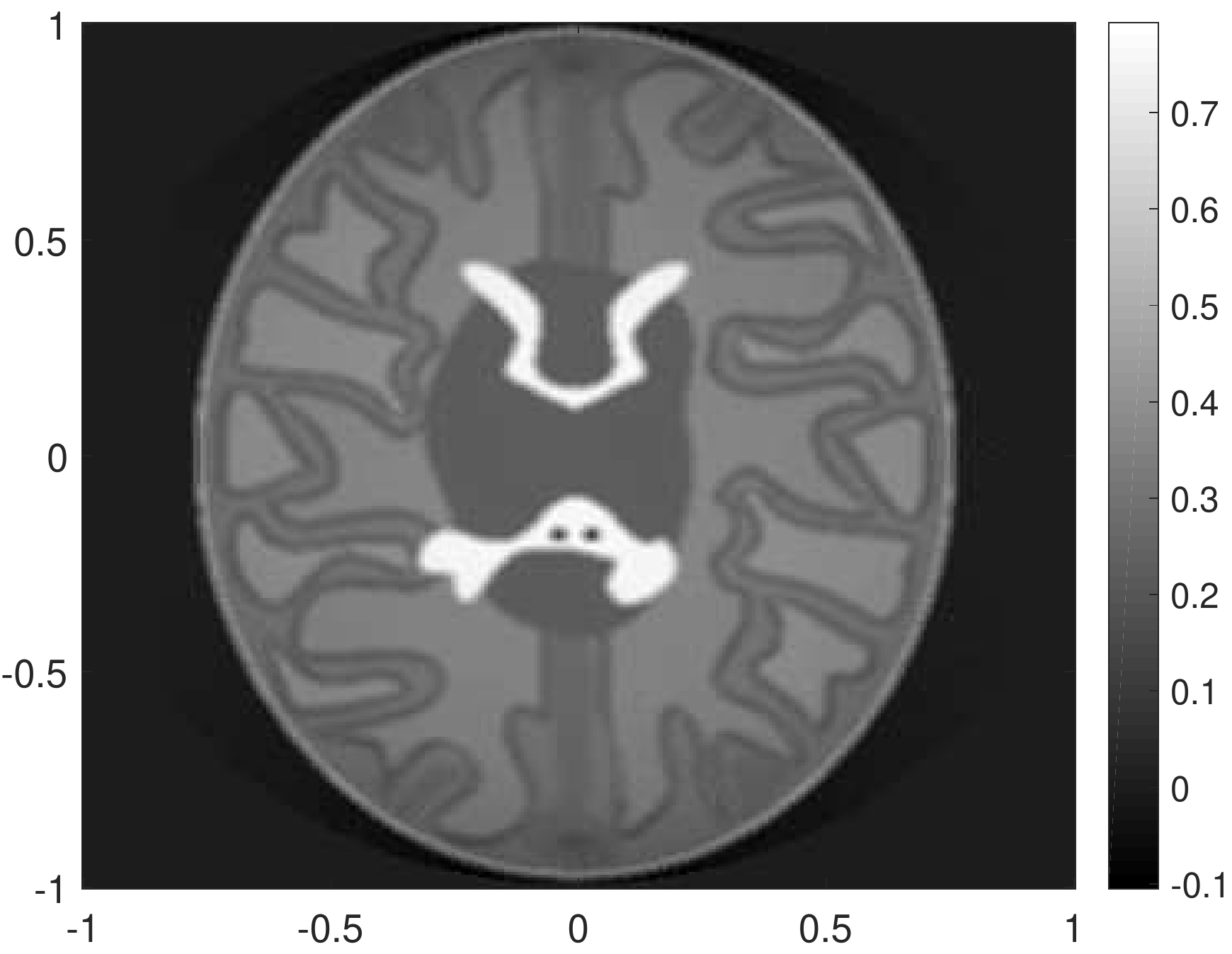} \\
			\includegraphics[width=6.5cm]{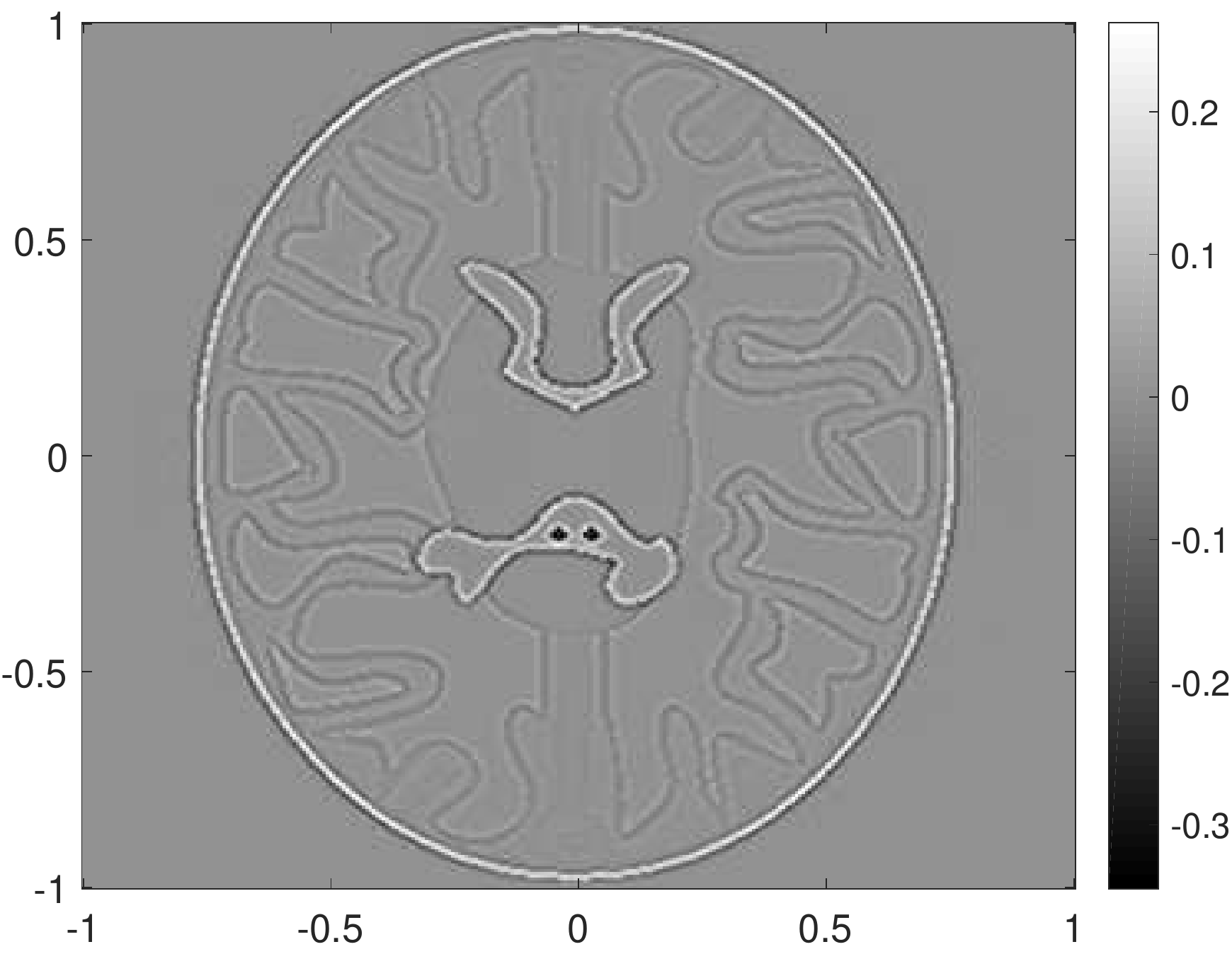}
			&
			\includegraphics[width=6.5cm]{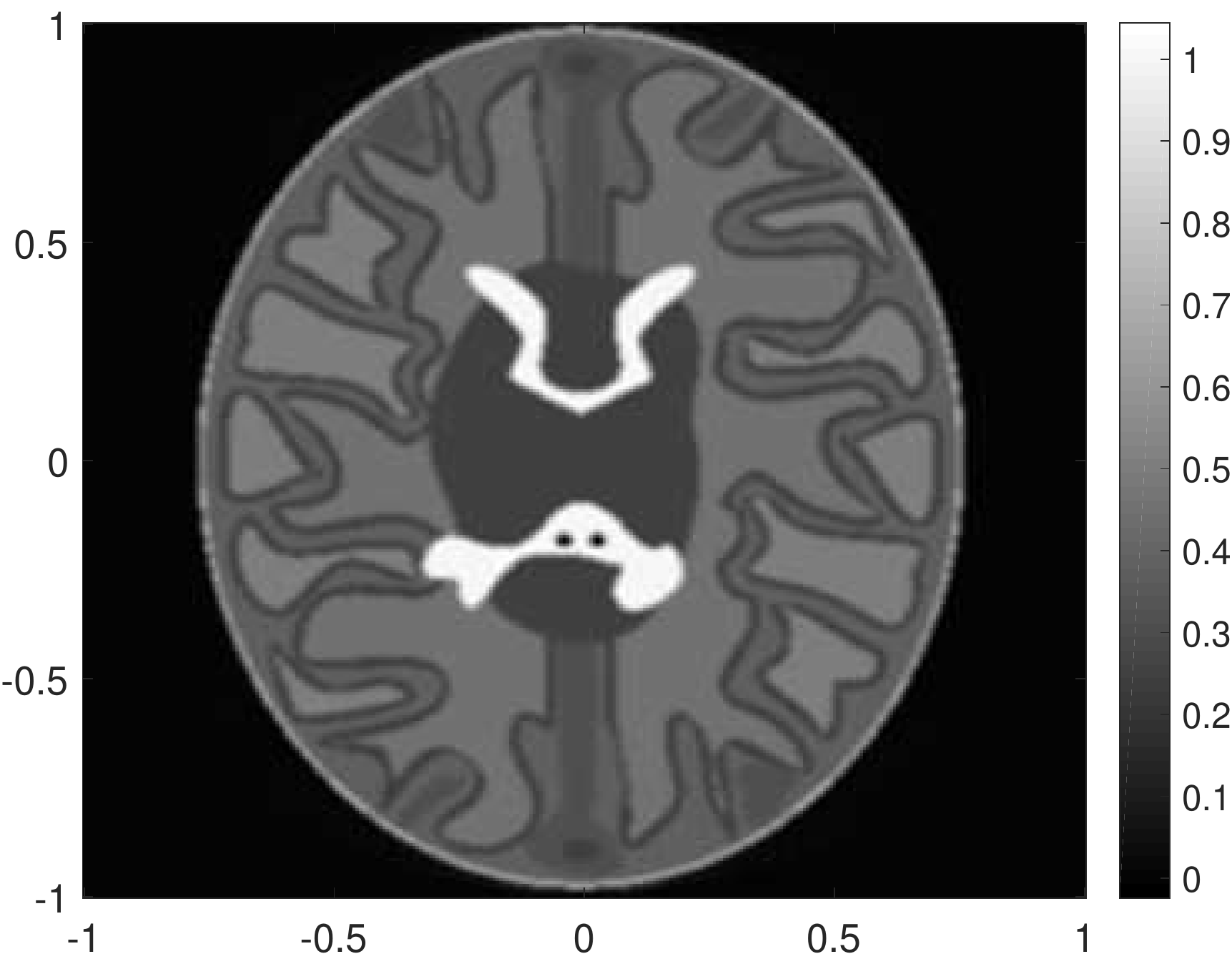}
		\end{tabular}
		\caption{Reconstructions with exact data: Top, left: $\mathbf{f_{\mathrm{rec,d}}}$ using Dirichlet measurements. Top, right: $\mathbf{f_{\mathrm{rec,n}}}$ using Dirichlet measurements. Middle, left: $\mathbf{f_{\mathrm{rec,d}}}$ using mixed measurements. Middle, right: $\mathbf{f_{\mathrm{rec,m}}}$ using mixed measurements. Bottom, left: $\mathbf{f_{\mathrm{rec,d}}}$ using Neumann measurements. Bottom, right: $\mathbf{f_{\mathrm{rec,n}}}$ using Neumann measurements.}
		\label{fig:rec_exact}
	\end{figure}
	In Figure \ref{fig:rec_exact} we observe that both implementations of our derived inversion formulas \eqref{eq:exactformnm} and \eqref{eq:exactformmixed} show very good results and approximate the head phantom almost perfectly. As we proved in Lemma \ref{lem:lemidentity3}, the numerical approximation of the integral (Figure \ref{fig:rec_exact} top, right) is close to zero. We point out that the numerical reconstruction $\mathbf{f_{\mathrm{rec,d}}}$ using mixed measurements is also quite good (Figure \ref{fig:rec_exact}, middle, left).
	\begin{figure}[h!]
		\centering
		\begin{tabular}{cc}
			\includegraphics[width=6.5cm]{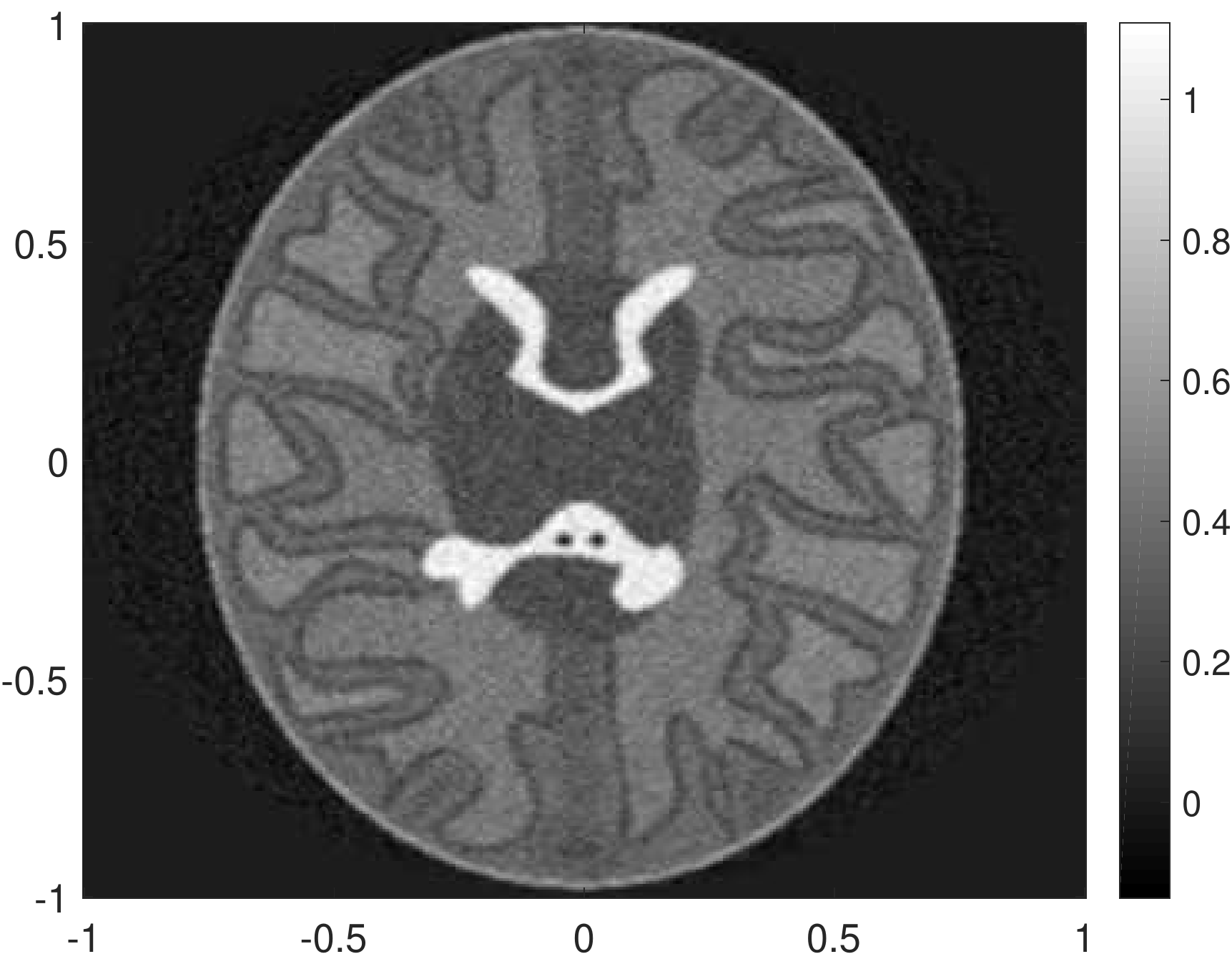}
			&
			\includegraphics[width=6.5cm]{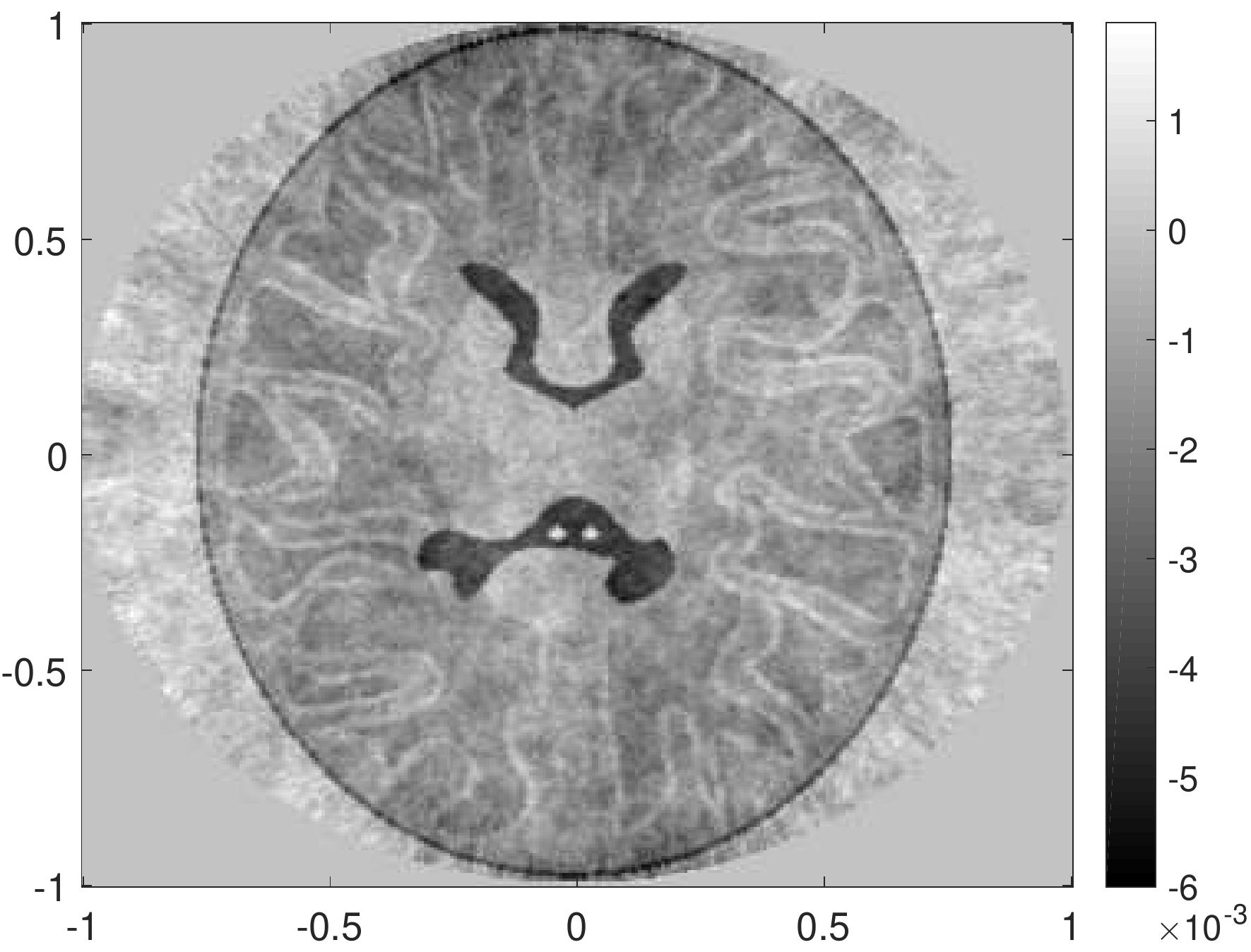} \\
			\includegraphics[width=6.5cm]{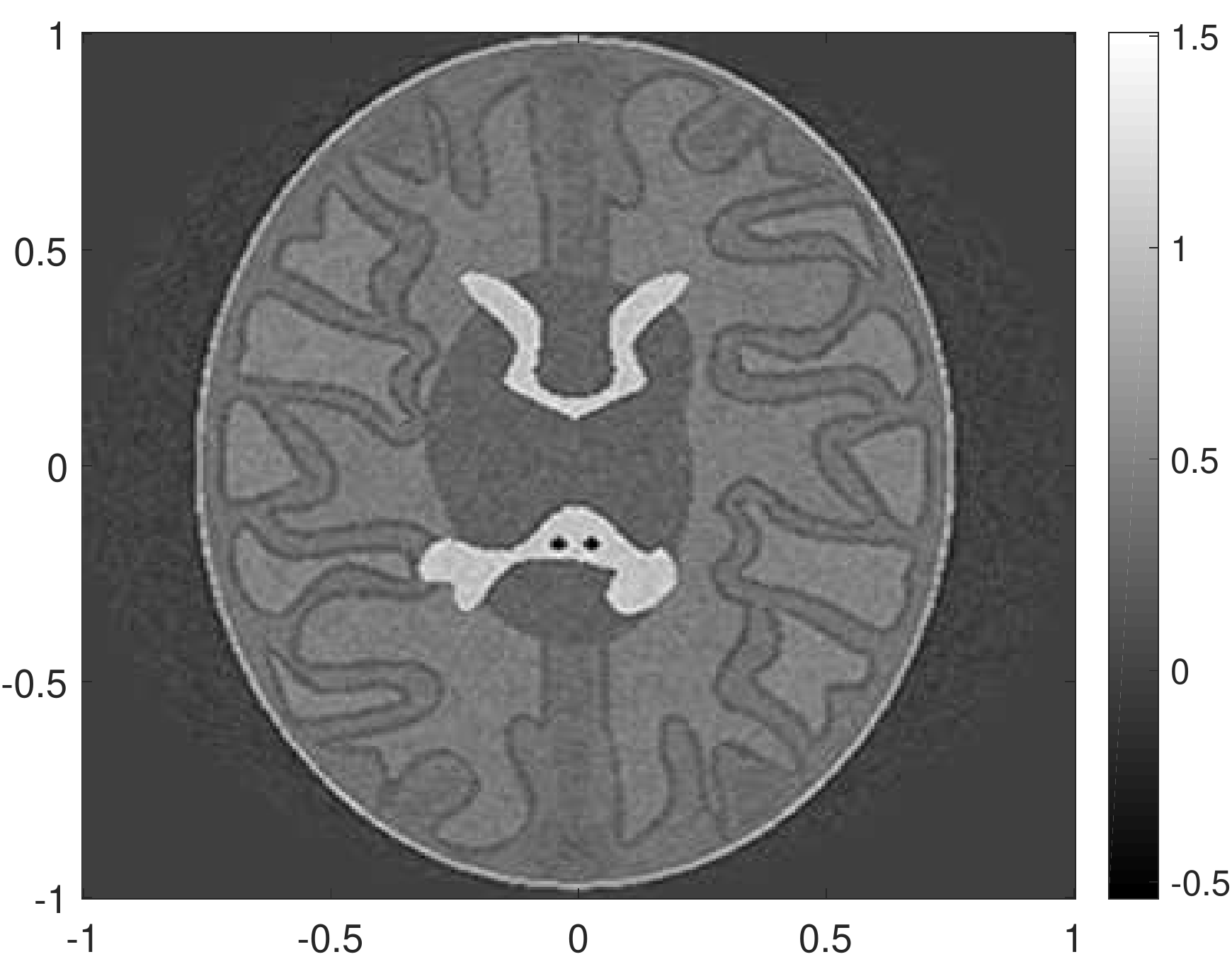}
			&
			\includegraphics[width=6.5cm]{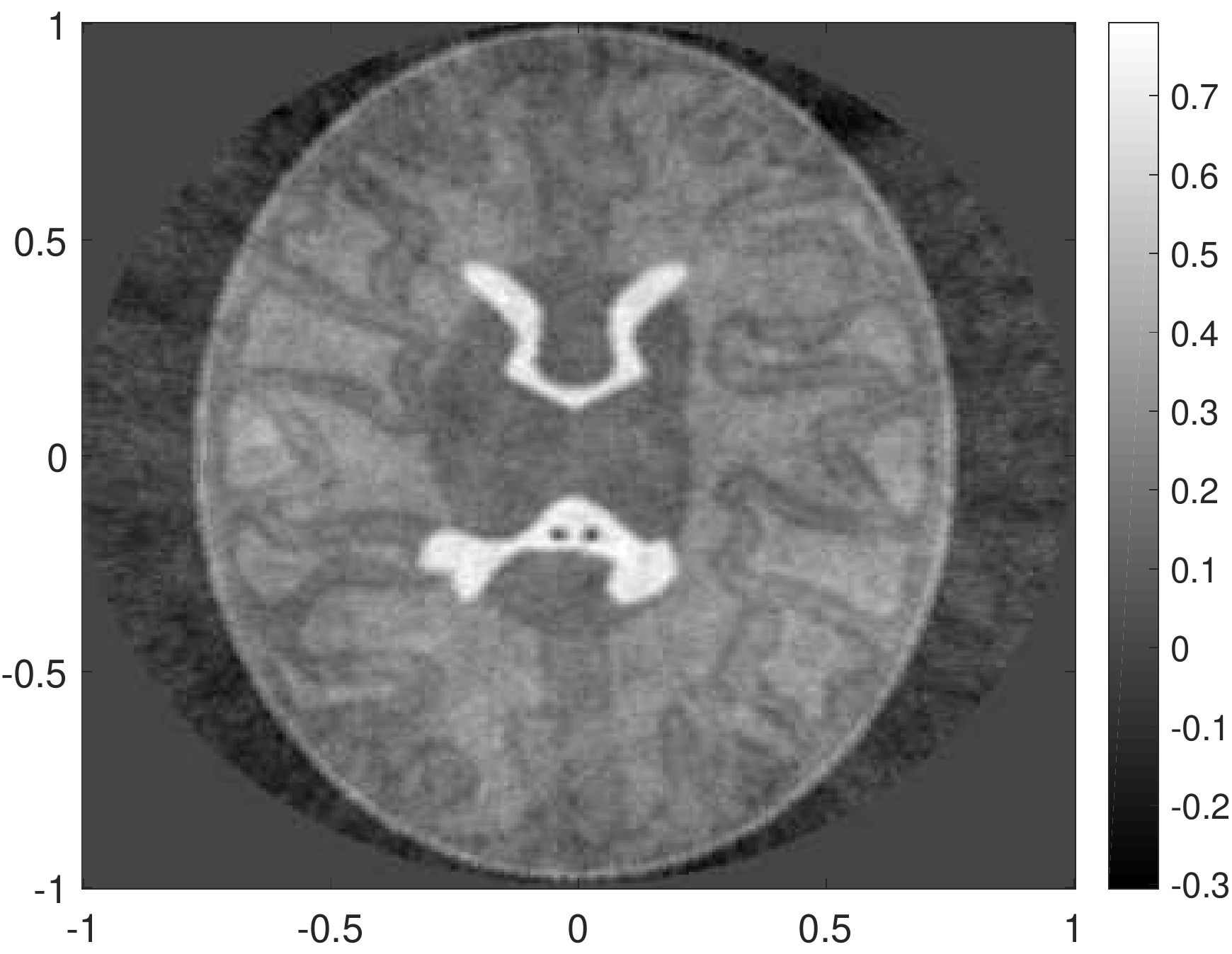} \\
			\includegraphics[width=6.5cm]{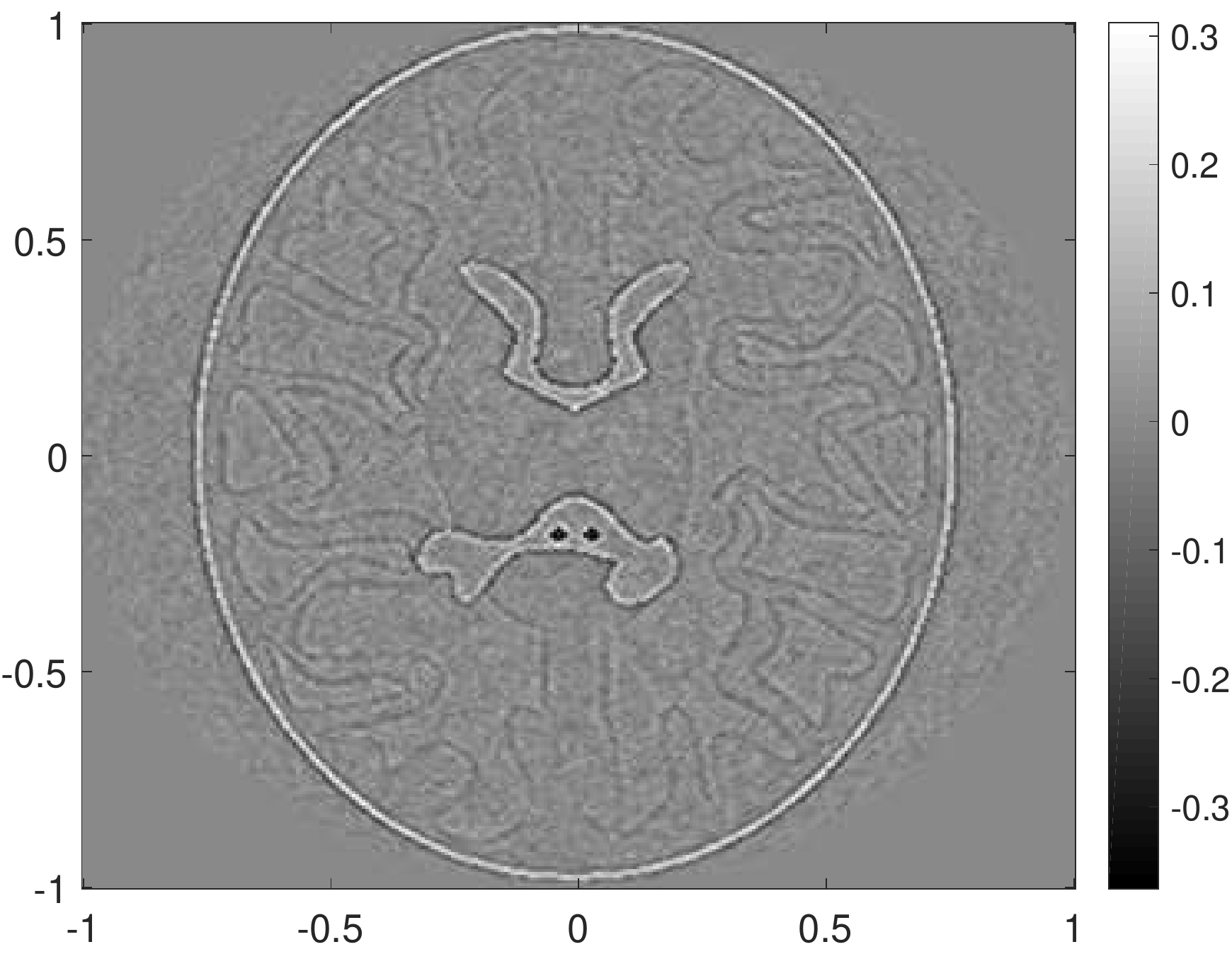}
			&
			\includegraphics[width=6.5cm]{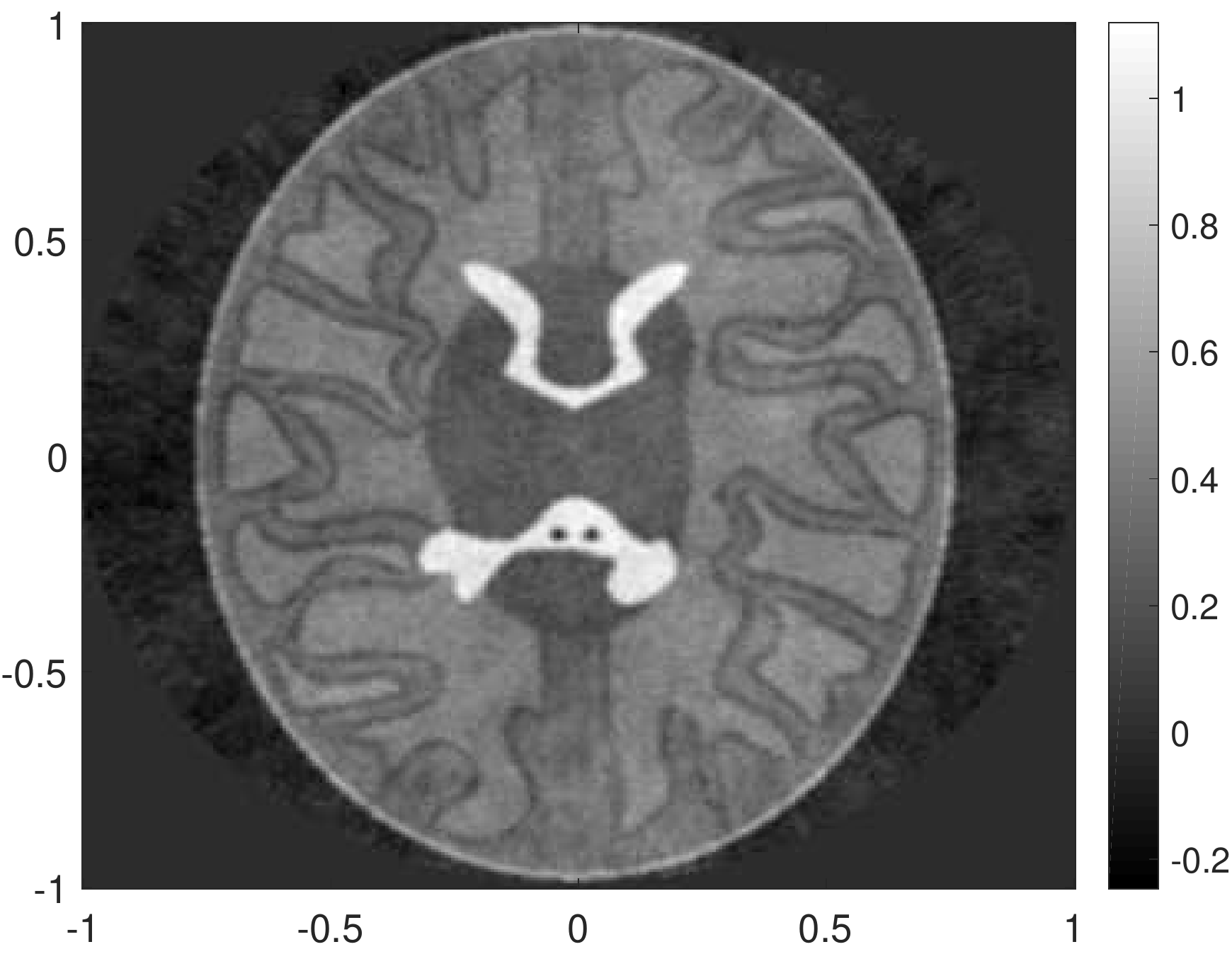}
		\end{tabular}
		\caption{Reconstructions with $10\%$ Gaussian noise added data: Top, left: $\mathbf{f_{\mathrm{rec,d}}}$ using Dirichlet measurements. Top, right: $\mathbf{f_{\mathrm{rec,n}}}$ using Dirichlet measurements. Middle, left: $\mathbf{f_{\mathrm{rec,d}}}$ using mixed measurements. Middle, right: $\mathbf{f_{\mathrm{rec,m}}}$ using mixed measurements. Bottom, left: $\mathbf{f_{\mathrm{rec,d}}}$ using Neumann measurements. Bottom, right: $\mathbf{f_{\mathrm{rec,n}}}$ using Neumann measurements.}
		\label{fig:rec_noisy10}
	\end{figure}
	\begin{figure}[h!]
		\centering
		\begin{tabular}{cc}
			\includegraphics[width=6.5cm]{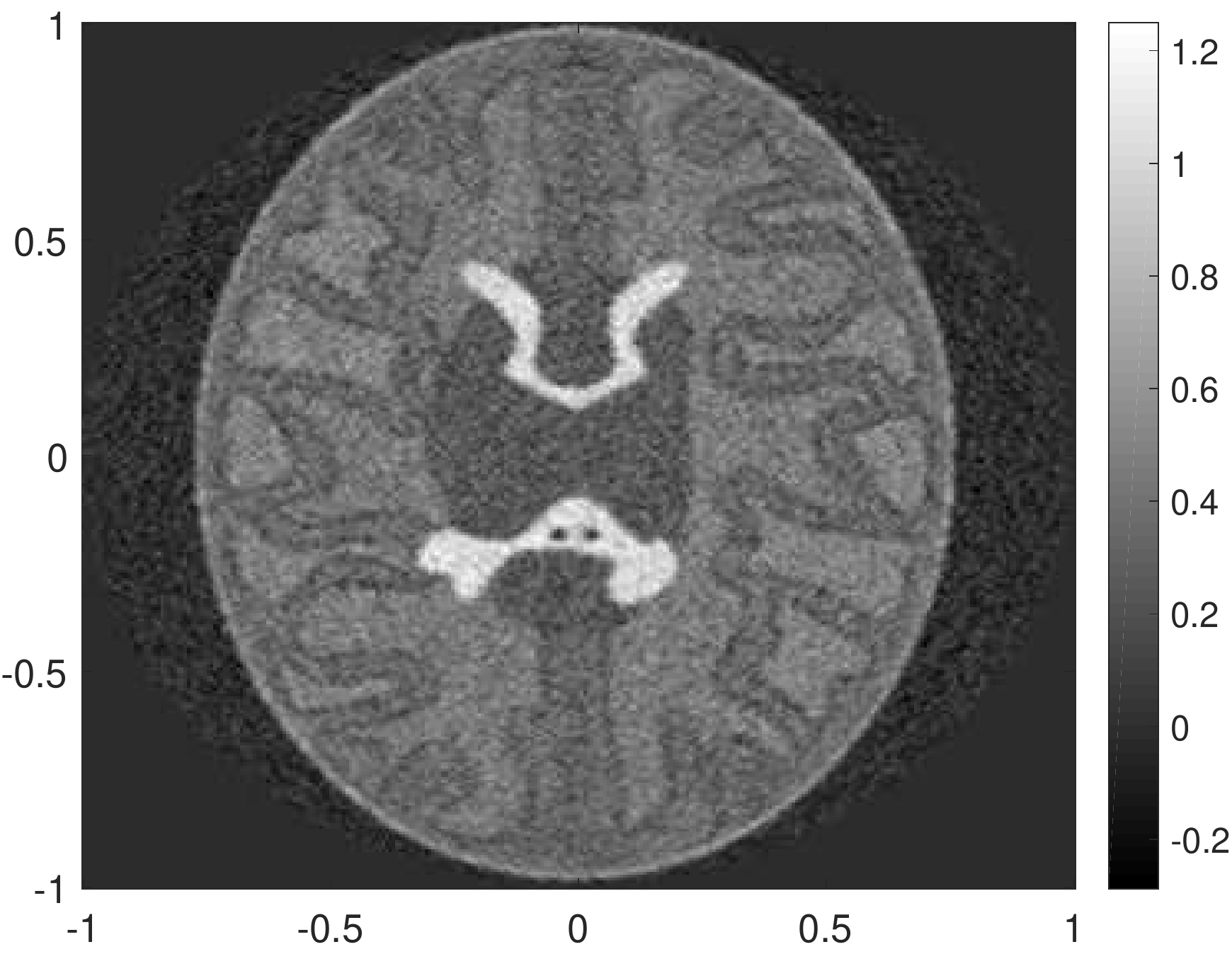}
			&
			\includegraphics[width=6.5cm]{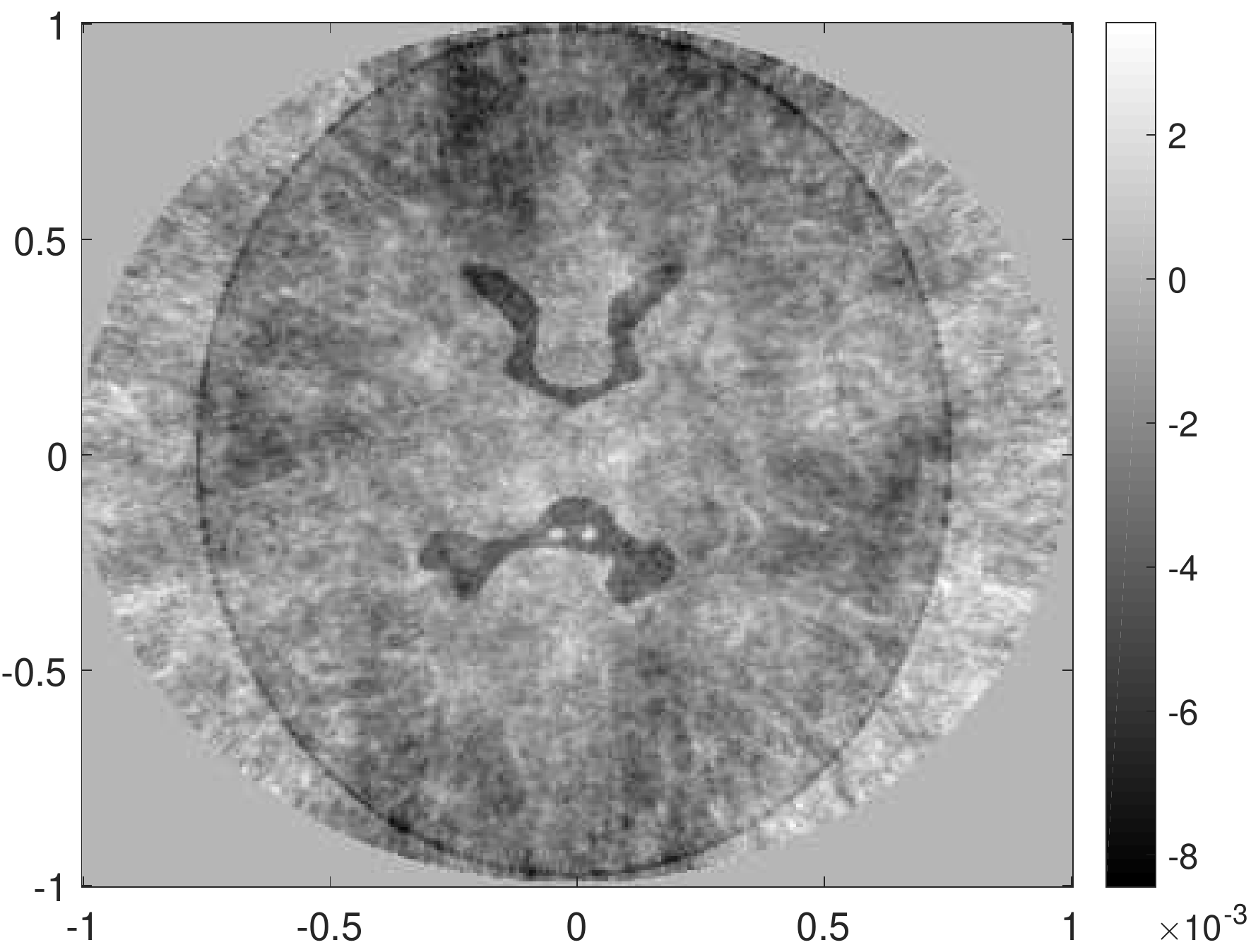} \\
			\includegraphics[width=6.5cm]{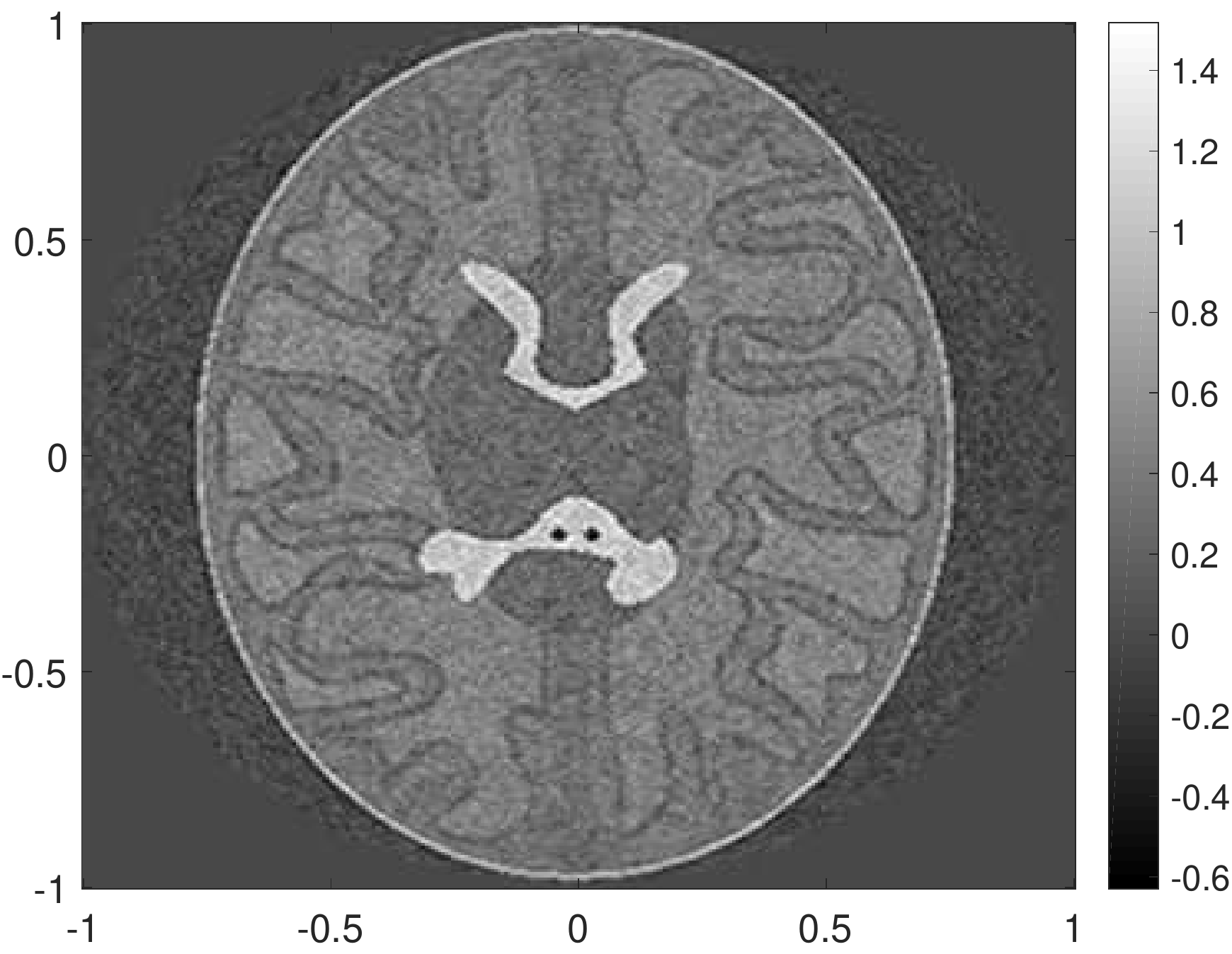}
			&
			\includegraphics[width=6.5cm]{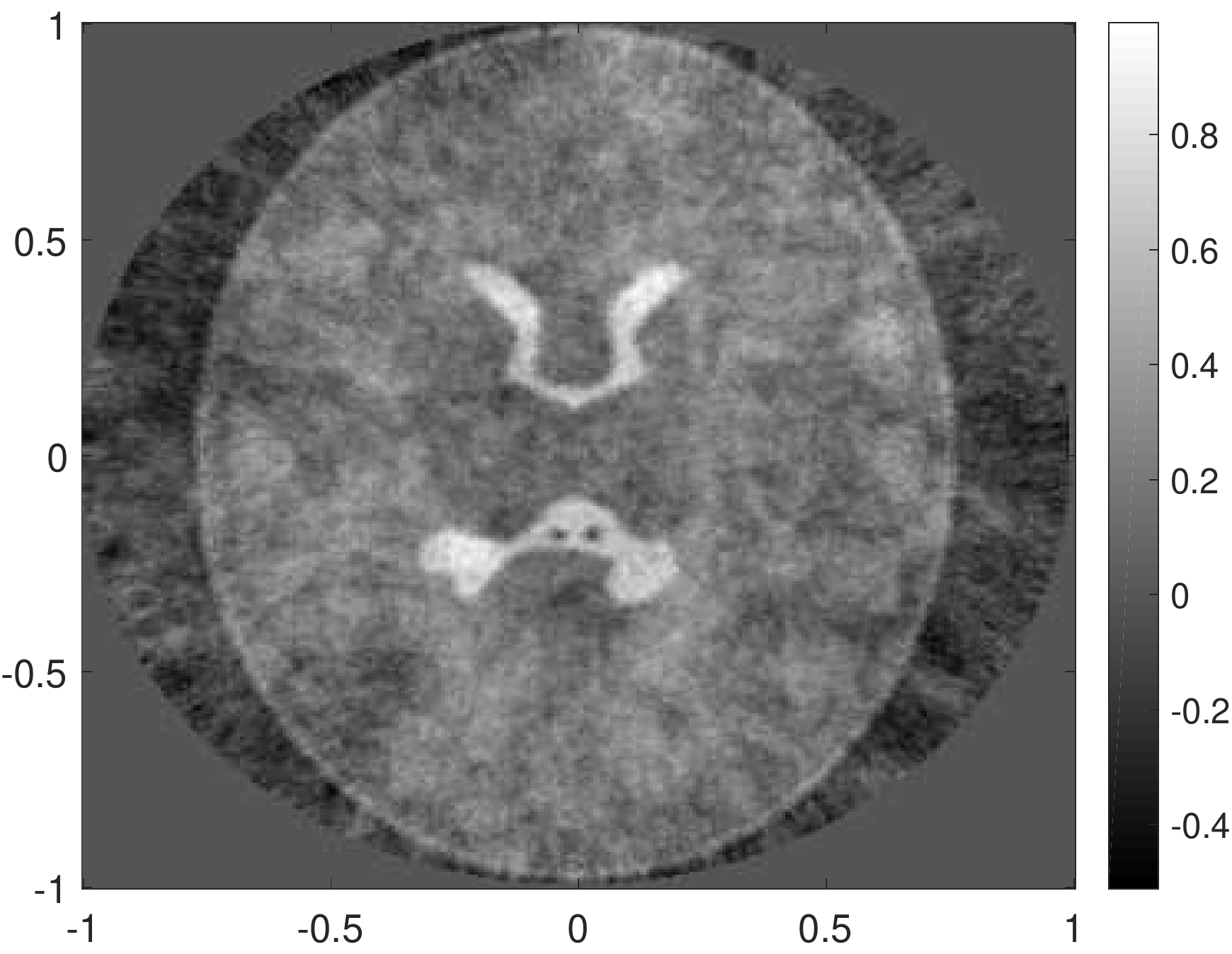} \\
			\includegraphics[width=6.5cm]{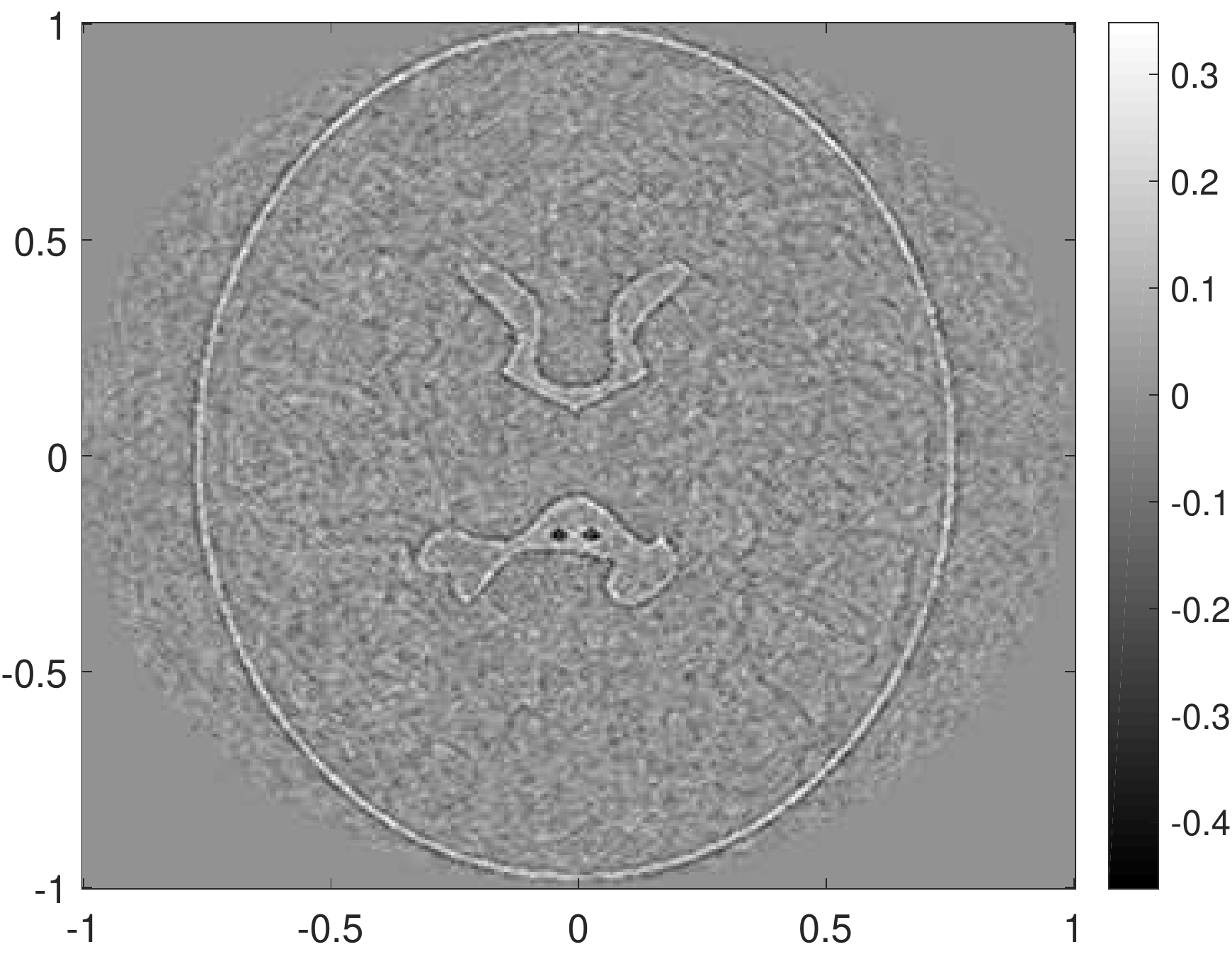}
			&
			\includegraphics[width=6.5cm]{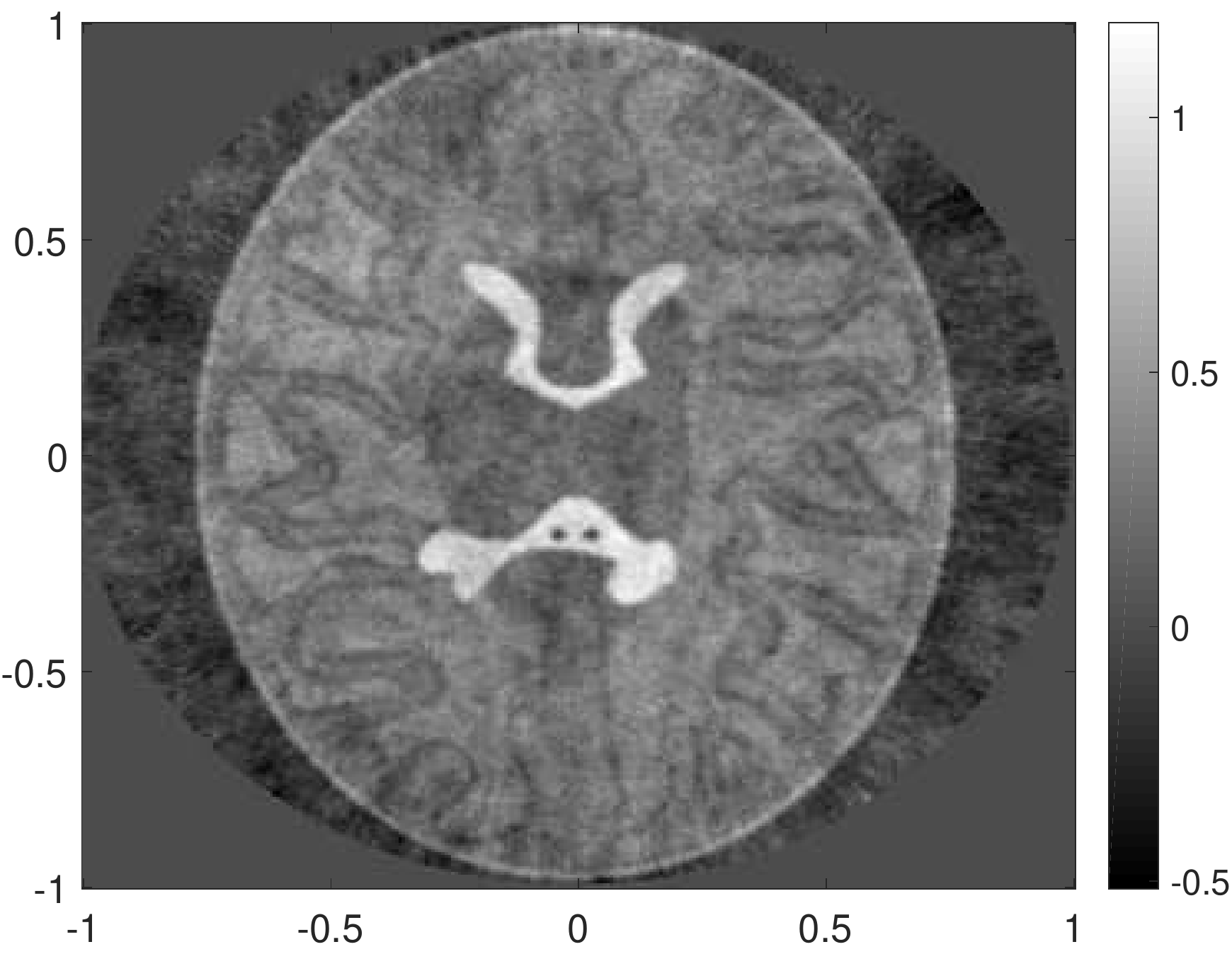}
		\end{tabular}
		\caption{Reconstructions with $20\%$ Gaussian noise added data: Top, left: $\mathbf{f_{\mathrm{rec,d}}}$ using Dirichlet measurements. Top, right: $\mathbf{f_{\mathrm{rec,n}}}$ using Dirichlet measurements. Middle, left: $\mathbf{f_{\mathrm{rec,d}}}$ using mixed measurements. Middle, right: $\mathbf{f_{\mathrm{rec,m}}}$ using mixed measurements. Bottom, left: $\mathbf{f_{\mathrm{rec,d}}}$ using Neumann measurements. Bottom, right: $\mathbf{f_{\mathrm{rec,n}}}$ using Neumann measurements.}
		\label{fig:rec_noisy20}
	\end{figure}
	Finally, we remark that the numerical reconstructions of the head phantom in Figure \ref{fig:rec_noisy10} show that the structure of the head phantom is still visible although Gaussian noise was added to the data. In Figure \ref{fig:rec_noisy20}, we also see that numerical implementation of formula \eqref{eq:ubp} with mixed measurements even shows a better result as the numerical implementation of the exact formula \eqref{eq:exactformmixed}. This leads to assumption that the inversion formula \eqref{eq:ubp} is also exact for mixed measurements.
	
	\section{Conclusion and Outlook}
	In this paper, we studied the problem of recovering the initial data of the two dimensional wave equation from Neumann traces. We established an explicit inversion formula for convex domains with smooth boundary up to an explicitly computed smoothing integral operator. This integral operator has been seen to vanish for circular and elliptical domains. We also derived an exact reconstruction formula for recovering the initial data from any linear combination of the solution of wave equation and its normal derivative on circular domains. The numerical results in the last section of this article showed that our implementation of the exact inversion formula leads to very good results even with noisy data.
	
In future work we intend to investigate the problem of determining the initial data of higher dimensional wave equations from Neumann measurements as well as from mixed measurements. We also want to establish explicit inversion formulas for Neumann data and mixed trace for other special domains such as certain quadratic hypersurfaces or even non-convex domains. Moreover, the derivation of inversion formulas from Neumann measurements on a finite time interval $\partial\Omega\times (0,T)$ instead of $\partial\Omega\times (0,\infty)$ could be also of great interest. 

%
\end{document}